\documentclass{amsart}

\usepackage{amssymb}
\usepackage{amsthm}
\usepackage{enumitem}
\usepackage{bbm}
\usepackage{calc}

\usepackage{graphicx}
\usepackage{comment}
\usepackage{tikz}
\usepackage{t1enc}
\usepackage{dsfont}
\usepackage{mathrsfs} 
\usepackage{mathtools} 
\usepackage{url}

\DeclareMathAlphabet{\mymathbb}{U}{bbold}{m}{n}

\newtheorem{thm}{Theorem}

\theoremstyle{definition}
\newtheorem{rem}[thm]{Remark}
\newtheorem{exam}[thm]{Example}
\newtheorem{lem}[thm]{Lemma}
\newtheorem{defn}[thm]{Definition}

\newtheorem{cor}[thm]{Corollary}

\newtheorem{fact}[thm]{Fact}
\newtheorem*{ques*}{Question}
\newtheorem*{prob*}{Problem}

\newcommand{\base}{r}

\newcommand{\cE}{\mathcal{E}}

\newcommand{\cM}{\mathcal{M}}

\newcommand{\cC}{\mathcal{C}}

\newcommand{\sB}{\mathcal{B}}

\newcommand{\wf}{\text{Wf}}
\newcommand{\tr}{\text{Tr}}

\newcommand{\bs}{\boldsymbol{\Sigma}}
\newcommand{\bp}{\boldsymbol{\Pi}}
\newcommand{\bP}{\boldsymbol{\Pi}}
\newcommand{\bS}{\boldsymbol{\Sigma}}
\newcommand{\bG}{\boldsymbol{\Gamma}}

\newcommand{\bsdzo}{{\check{\boldsymbol{\Sigma}}}^{\raisebox{-2pt}{\scriptsize{0}}}_1}
\newcommand{\bsdza}{{\check{\boldsymbol{\Sigma}}}^{\raisebox{-2pt}{\scriptsize{0}}}_\alpha}

\newcommand{\bd}{\boldsymbol{\Delta}}

\newcommand{\res}{\restriction}

\newcommand{\ww}{\omega^\omega}

\DeclareMathOperator{\Emp}{\mathcal{E}}

\DeclareMathOperator{\seq}{seq}

\renewcommand{\phi}{\varphi}

\DeclareMathOperator{\M}{\mathcal{M}}
\DeclareMathOperator{\MT}{\mathcal{M}_{\mathit T} }

\DeclareMathOperator{\MTe}{\M_{\mathit T}^e}

\newcommand{\weakst}{weak$^*$}

\newcommand{\eps}{\varepsilon}
\newcommand{\R}{\mathbb{R}}

\newcommand{\N}{\mathbb{N}}

\author[K. Deka]{Konrad Deka}%
\address[K. Deka]{
Faculty of Mathematics and Computer Science, Jagiellonian University in Krakow, ul. \L o\-jasiewicza 6, 30-348 Krak\'ow, Poland,
}
\email{Konrad.Deka}
\author[S. Jackson]{Steve Jackson}
\address[S. Jackson]{Department of Mathematics, University of North Texas,
General Academics Building 435, 1155 Union Circle,  \#311430, Denton, TX 76203-5017, USA}
\email{stephen.jackson@unt.edu}
\author[D. Kwietniak]{Dominik Kwietniak}%
\address[D. Kwietniak]{
Faculty of Mathematics and Computer Science, Jagiellonian University in Krakow, ul. \L o\-jasiewicza 6, 30-348 Krak\'ow, Poland,
}
\email{dominik.kwietniak@uj.edu.pl}
\urladdr{www.im.uj.edu.pl/DominikKwietniak/}

\author[B. Mance]{Bill Mance}
\address[B. Mance]{Faculty of Mathematics and Computer Science, Adam Mickiewicz University, ul. Umultowska 87, 61-614 Pozna\'{n}, Poland}
\email{William.Mance@amu.edu.pl}
\urladdr{wilman.home.amu.edu.pl/}

\title[Borel complexity of sets of orbits]
{Borel complexity of sets of points with prescribed Birkhoff averages in Polish dynamical systems with a specification property}

\begin{document}

\begin{abstract}We study the descriptive complexity of sets of points
defined by placing restrictions on statistical behaviour of their orbits in dynamical systems on Polish spaces. A particular examples of such sets are the set of generic points of a $T$-invariant Borel probability measure,  but we also consider much more general sets (for example, $\alpha$-Birkhoff regular sets and the irregular set appearing in multifractal analysis of ergodic averages of a continuous real-valued function). We show that many of these sets are Borel. In fact, all these sets are Borel when we assume that our space is  compact. We provide examples of these sets being non-Borel, properly placed at the first level of the projective hierarchy (they are complete analytic or co-analytic). This proves that the compactness assumption is in some cases necessary to obtain Borelness. When these sets are Borel, we use the Borel hierarchy to measure their descriptive complexity. We show that the sets of interest are located at most at the third level of the hierarchy. We also use a modified version of the specification property to show that for many dynamical systems these sets are properly located at the third level. To demonstrate that the specification property is a sufficient, but not necessary condition for maximal descriptive complexity of a set of generic points, we provide an example of a compact minimal system with an invariant measure whose set of generic points is $\bp^0_3$-complete.
\end{abstract}

\maketitle

\section{Introduction}
Every Borel subset of an uncountable Polish metric space $X$ is a result of the following inductive procedure, whose steps are enumerated by countable ordinals:
for $\alpha=1$ we take all  open subsets of $X$ and their complements (closed subsets of $X$). We continue by transfinite induction:
if $\beta<\omega_1$ is a countable ordinal and we completed our procedure for all ordinals $0<\alpha<\beta$, then at the step $\beta$ we add all possible countable unions (or intersections)
of sets obtained so far and their complements. This completes level $\beta$ of the procedure. So at the step $\alpha=2$ we add all $G_{\delta}$ and $F_\sigma$ subsets, which are neither open nor closed, and in the next step we obtain all sets which are countable intersections  of $F_\sigma$ sets and all countable unions of $G_\delta$ sets, which are neither $G_\delta$ nor $F_\sigma$ themselves. This introduces a natural hierarchy of Borel subsets of $X$ which reflects their
descriptive complexity. It is known that for an uncountable Polish space
these levels do not collapse: at each level there appear new sets (sets that do
not belong at any lower level of the hierarchy). Determining the level where a ``naturally arising'' or ``non-ad hoc''
Borel set appears in the hierarchy becomes a challenging problem. Only a few concrete examples of Borel sets are known to be located above the third level.

Here, we study the Borel complexity of sets appearing naturally in ergodic theory, dynamics, number theory, and fractal geometry. We develop tools which allow us to determine the exact position of many Borel sets examined in the literature. We show that without compactness some sets studied in topological dynamics become completely analytic (completely analytic sets are, in some sense, the simplest instances of non-Borel sets).

We state our main result in the framework of a dynamical system given by iteration of a continuous map $T$ on a Polish metric space $X$. In this setting the sets of interest consist of points whose orbits have prescribed statistical behaviour with respect to a bounded continuous observable $\phi\colon X\to\R$. Here by ``statistical behaviour of an orbit of $x\in X$'' we mean the asymptotic behaviour of the sequence of Birkoff averages of $\phi$ along the orbit of $x$. In the classical topological dynamics, one usually assumes that the space $X$ is compact, but some expansions, most notably, continued fraction expansions come from dynamical systems acting on non-compact Polish spaces, so we work in this greater generality.
We refer the reader to \cite{GM, ITV,IV, Sarig} for related problems for dynamical systems (linear operators, countable Markov chains) on non-compact phase spaces. We study the following questions:

\begin{prob*}
Given a continuous map $T$ from a Polish space $X$ to itself and a $T$-invariant
measure $\mu$, what is the complexity of the set $G_\mu$ of
generic points for the system ($G_\mu$ is defined precisely below)?
More generally, given a
nonempty, closed, and connected subset $V$ of the space $\MT(X)$ of
all $T$-invariant Borel probability measures on $X$ endowed with the
weak$^*$ topology, what is the Borel complexity of the set $G(V)$ of all
points $x$ in $X$ whose orbits generate $V$, that is, the set of limit
points of a sequence $\Emp(x,N)$ equals $V$, where $\Emp(x,N)$ denotes
the normalized counting measure concentrated on
$x,T(x),\ldots,T^{N-1}(x)$. Finally, given a bounded continuous function $\phi\colon X\to\R$, we study Borel complexity of sets of those $x\in X$ that
the sequence of Birkhoff averages
\[
A_k\phi(x)=\frac1k\sum_{j=0}^{k-1}\phi(T^j(x))
\]
converge as $k\to\infty$ to a given value $\alpha\in\R$. We also examine the set of points where this sequence does not converge for some or all bounded continuous function(s) (the \emph{irregular set}).
\end{prob*}

Our initial interest in these questions comes from a result of H.\ Ki and T.\ Linton, who answered a question posed by A.~Kechris
and showed in \cite{KiLinton} that for any
integer $\base \geq 2$, the set of numbers that are normal in base $\base$ is
a $\bp^0_3$-complete set (a $\bp^0_3$ set is a countable intersection of $F_\sigma$ sets,
the general definitions are given in  \S\ref{sec:voc}). Recall that we say that a real
number $\xi$ is {\em normal in base $\base$} if in its $\base$-adic expansion
every block of digits of length $k$ appears with asymptotic frequency
$1/\base^k$. After \cite{KiLinton} many authors have studied
the Borel complexity of various sets related to normal numbers, and
have extended Ki and Linton's result in different directions.
V.\ Becher, P.\ A.\ Heiber, and T.\ A.\ Slaman \cite{BecherHeiberSlamanAbsNormal}
settled a conjecture of A.\ S.\ Kechris by showing that the set of
absolutely normal numbers (a real number is {\em absolutely normal}
if it is normal to all bases $\base\ge 2$)
is $\bp^0_3$-complete. Furthermore, V.\ Becher and T.\ A.\ Slaman \cite{BecherSlamanNormal}
proved that the set of numbers normal in at least one base is
$\bs^0_4$-complete. See also \cite{AireyJacksonManceComplexityNormalPreserves}, \cite{BerosDifferenceSet}.
Also, sets of interest in dynamics are sometimes complete at levels past the Borel hierarchy.
For example, in \cite{JMR} a dynamical system in the unit square is constructed for which the set of
special $\alpha$-limit points is $\bP^1_1$-complete.

It is well-known that a real number $\xi$ being normal in base $\base$ is equivalent to
the sequence $(\base^k \xi)_k $ being uniformly distributed  mod $1$
(cf.\ \cite{KuN}). The latter statement can be rephrased as $\xi\bmod 1$ is a generic point for the Lebesgue measure on the unit circle $\mathbb{T}=\mathbb{R}/\mathbb{Z}$ under the action of the map $\mathbb{T}\to\mathbb{T}$ given by $x\mapsto \base x\bmod 1$.
Thanks to this dynamical point of view, in \cite{AJKM} we were able to study the Borel complexity of sets of normal numbers in several numeration systems. We employed
a unified treatment for $\base$-ary expansions, continued fraction expansions, $\beta$-expansions, and generalized GLS-expansions. In fact, we considered a dynamical system given by subshifts generated by all possible expansions of numbers in $[0,1]$ in full-shifts over an at most countable set of digits (alphabet). In that setting normal numbers with respect to a given expansion method correspond to generic points of an appropriately selected shift-invariant measure on a subshift containing all possible expansions. It turns out that for these subshifts the set of generic points for any shift-invariant probability measure is precisely at the third level of the Borel hierarchy (it is a $\bp^0_3$-complete set). The crucial dynamical feature we used was a feeble form of specification. All expansions named above generate subshifts with this specification property.

Our aim is to extend the results of \cite{AJKM}. We generalise \cite{AJKM} in two directions: First, we consider arbitrary (in particular, non-symbolic)  dynamical systems on Polish spaces. In \cite{AJKM} only symbolic systems were considered. Second, we determine the descriptive complexity of sets of points defined by placing restrictions on statistical distribution of their orbits with respect to all bounded continuous observables. By statistical distribution of the orbits we understand the asymptotic behaviour of a sequence of Birkhoff averages along an orbit. In \cite{AJKM} only sets of generic points were considered. Our level of generality still allows us to study dynamical systems leading to various expansions considered  in \cite{AJKM}, but it also contains many more examples: we are now able to study the descriptive complexity of sets of numbers defined in terms of the asymptotic behaviour of the frequencies of the digits in their expansions. So far, only of the Hausdorff dimension of these sets was examined (see, for example, L. M. Barreira, B. Saussol and J. Schmeling, \cite{BSS}, A. H. Fan, D. J. Feng and J. Wu, \cite{FFW}, or L. O. R. Olsen, \cite{Olsen}). To apply our results in other branches of mathematics it remains  to explain how to reduce the set we want to examine to the relevant dynamical form.



There are very simple examples of compact dynamical systems with ergodic invariant measures whose generic points are Borel sets of low complexity:
If $X$ is the unit circle and $T$ is an irrational rotation, then every point is generic for the Lebesgue measure $\lambda$ so $G_\lambda$ is clopen. On the other hand, if $X=[0,1]$ and $T(x)=x^2$ for $0\le x \le 1$, then $T$ has two ergodic invariant measures, namely Dirac-point measures $\delta_0$ and $\delta_1$ concentrated at fixed points $0$ and $1$. We clearly have
$G_{\delta_0}=[0,1)$ is open but not closed, and $G_{\delta_1}=\{1\}$ is closed but not open. Hence, similarly as in \cite{AJKM} some assumptions about the dynamical system are necessary to achieve the maximal complexity of sets with prescribed Birkhoff averages for some (or all) observables.
It turns out that the right assumption to
impose on the system is a form of the {\em specification property}.
The specification property was introduced by R.\ Bowen in his
paper on Axiom A diffeomorphisms \cite{BowenAxiomA}. However, for the main results of our paper
only a certain weak form of the specification property (coined the \emph{strong approximate product structure}) is needed.
It implies the feeble specification property used in \cite{AJKM} for symbolic systems.
We refer the reader to \cite{KLO} for a discussion of the specification property
and its many variants as well as their significance in dynamics.

We study at which Borel hierarchy levels the sets $G_\mu$ (or more
generally, sets like $G(V)$, see \S\ref{sec:SOI-def} for the full list) may appear in the above situation and present a
theorem saying that under fairly general assumptions these sets always
have the maximal possible complexity.

\subsection*{Organization of the paper}
After reviewing the notation and terminology in \S\ref{sec:voc},
we discuss the specification property and its weakening in \S\ref{sec:spec}. In  \S\ref{sec:SOI-def} we define
various sets described by the statistical behaviour of Birkhoff averages, and in \S\ref{sec:ub} we prove that all these sets are (co-)analytic, some are always Borel, and all are always Borel when the space $X$ is compact. In other words, in \S\ref{sec:ub} we gather upper bounds for the descriptive complexity of our sets of interest.
These computations hold in general, that is, they do not require specification hypotheses on the system.
In \S\ref{sec:ws} we state and prove our main results regarding lower bounds assuming our variant of the specification property (the strong approximate product structure).
Whenever these bounds match the upper bounds obtained in the previous section, we obtain a completeness result.
In particular, this is the case for all these sets when $X$ compact.
For example, we see that the set of generic points $G_\mu$  is $\bp^0_3$-complete. 
In \S\ref{sec:hlc} we present a (non-compact) dynamical system for which $G(V)$ (here
$V$ may consist of a single measure $\mu$) is $\bp^1_1$-complete.
Thus, we see a dichotomy in the complexity of $G(V)$ depending on whether $X$ is
compact. Finally, in \S\ref{sec:minimalsubshift} we give an example of a minimal compact symbolic system (subshift) $(X, T)$ with an invariant ergodic measure $\mu$ such that $G_\mu$ is $\bP_3^0$-complete. This example shows that the specification property is sufficient for maximal complexity, but by no means necessary.

\subsection*{Acknowledgments}

Konrad Deka and Dominik Kwietniak were supported by the National Science
Centre, Poland, grant Preludium-bis no. 2019/35/O/ST1/02266. Steve Jackson was supported by NSF grant 1800323. Bill Mance is supported by the National Science Centre, Poland, grant Sonata-bis no.  2019/34/E/ST1/00082 for the project.

\section{Vocabulary/definitions/notation} \label{sec:voc}
\subsection{General} Throughout this paper $\omega=\{0,1,2,\ldots\}$, $\N$ stands for the set of positive integers,
the variables $i,j,k,\ell,m,n$ are integers, and
$X$ always denotes a Polish topological space.
Note that we are not assuming compactness
of $X$ in general; any compactness assumptions on $X$ will be stated explicitly
as needed (our main result does not require any compactness assumptions).

By a \emph{Polish dynamical system}, which we also just call a
\emph{dynamical system}, we mean a pair $(X,T)$ where
$T\colon X\to X$ is a continuous map.
We stress that whenever we consider a Polish dynamical system $(X,T)$ we tacitly assume that $X$ is equipped with a fixed  compatible complete metric $\rho$. If $X$ is non-compact, then various properties we will be considering, such as the specification property
and its variants, depend on $\rho$ in a sense that changing the metric to an equivalent one may affect these properties. For an example where the specification property is lost with a change of metric, see \cite[Proposition 7.1]{KU} (in other words, for non-compact spaces the specification property is no longer an invariant for topological conjugacy). For compact spaces the choice of metric is irrelevant.


\subsection{Borel, analytic, and co-analytic sets}
We recall some basic notions from descriptive set theory which we use
to measure  the complexity of sets in Polish spaces.
The collection of {\em Borel sets} $\sB(X)$, for $X$ a topological space, is
the smallest $\sigma$-algebra containing the open sets.

We let $\bs^0_1$ be the collection of open sets and  $\bp^0_1=
\bsdzo= \{ X\setminus A\colon A \in \bs^0_1\}$ be the collection of closed sets.
In general, for $\alpha<\omega_1$ we let $\bs^0_\alpha$ be the collection of
countable unions $A=\bigcup_n A_n$ where each $A_n \in \bp^0_{\alpha_n}$
for some $\alpha_n<\alpha$. We also let $\bp^0_\alpha= \bsdza= \{ X\setminus A\colon A \in \bs^0_\alpha\}$.
Equivalently, $A\in \bp^0_\alpha$ if $A=\bigcap_n A_n$ where
$A_n\in \bs^0_{\alpha_n}$ and each $\alpha_n<\alpha$.
We also set $\bd^0_\alpha= \bs^0_\alpha \cap \bp^0_\alpha$. In particular,
$\bd^0_1$ is the collection of clopen sets.
In classical terminology, $\bs^0_2$ is the collection of $F_\sigma$ sets and $\bp^0_2$
is the collection of $G_\delta$ sets.
For any topological space, $\sB(X)=\bigcup_{\alpha<\omega_1} \bs^0_\alpha=
\bigcup_{\alpha<\omega_1}\bp^0_\alpha$. A basic fact (see \cite{Kechris})
is that for any uncountable Polish space $X$, all of the classes $\bd^0_\alpha$, $\bs^0_\alpha$, $\bp^0_\alpha$, for $\alpha <\omega_1$,
are  distinct.  We say a set
$A\subseteq X$ is $\bs^0_\alpha$ (resp.\ $\bp^0_\alpha$) {\em hard}
if $A \notin \bp^0_\alpha$ (resp.\ $A\notin \bs^0_\alpha$). This says that $A$ is
``no simpler'' than a $\bs^0_\alpha$ set. We say that $A$ is $\bs^0_\alpha$-{\em complete}
if $A\in \bs^0_\alpha\setminus \bp^0_\alpha$, that is, $A \in \bs^0_\alpha$ and
$A$ is $\bs^0_\alpha$ hard. This says that $A$ is exactly at the complexity level
$\bs^0_\alpha$. Likewise, $A$ is $\bp^0_\alpha$-complete if $A\in \bp^0_\alpha\setminus \bs^0_\alpha$.

The Borel sets lie at the base of another hierarchy, the
{\em projective hierarchy}. The $\bs^1_1$ (or analytic) sets are the images
 of Borel sets (via continuous functions from a Polish space to a Polish space). The
same collection results if one uses Borel images instead of continuous images  and
one can also take just images of closed sets. The $\bp^1_1$ (or co-analytic) sets
are the complements of the analytic sets.

We define $\bd^1_1=\bs^1_1\cap \bp^1_1$.
The pointclasses $\bs^1_1$ and $\bp^1_1$ are closed under countable unions and intersections
and so contain all the Borel sets.
A fundamental result of Suslin (see \cite{Kechris}) says that in any
Polish space $\sB(X)=\bd^1_1$. One can continue to define the
$\bs^1_n$ and $\bp^1_n$ sets for $n \geq 2$, forming the {\em projective hierarchy},
but we will not consider any
classes beyond $\bp^1_1$and $\bs^1_1$ in this paper.

Also, all of the collections
$\bd^0_\alpha$, $\bs^0_\alpha$, $\bp^0_\alpha$, $\bp^1_1$, and $\bs^1_1$ are {\em pointclasses}, that is, they are closed
under inverse images of continuous functions.

The levels of the Borel (and projective) hierarchy can be used
to calibrate the descriptive complexity of a set.
Determining an upper bound on the complexity of the set $A$ generally involves
writing a condition defining $A$
which shows that it is in the appropriate place in the hierarchy.
Establishing lower bounds is generally more difficult.
One technique is the method of ``Wadge
reduction'' whereby a set $C\subseteq Y$ ($Y$ a Polish space)
which is known to be hard for a certain
level $\bG$ of the hierarchy is ``reduced'' to the set $A$. By this we mean
we have a continuous function $f\colon Y\to X$ such that
$C=f^{-1}(A)$. Since $\bG$ is a pointclass, this shows that $A$
is $\bG$-hard.

\subsection{Invariant measures, generic points}
We denote the set of all Borel probability measures on $X$ by $\M(X)$ and equip it with the \weakst\ topology, which in our setting is known to be Polish
\cite[\S17.E]{Kechris}.
Recall that $\mu_n$ converges to $\mu$ in the \weakst\ topology
if and only if for every $f \in C_b(X)$
(i.e., $f$ is a bounded, continuous,
real-valued function on $X$) we have $\int f\, d\mu_n \to \int f\, d\mu$.
If $X$ is compact, then $\M(X)$ is also compact \cite[Theorem 17.22]{Kechris}.
Although our main results are of a topological nature, it is convenient for us to fix once
and for all a particular metric $D$ for $\M(X)$. We choose the \emph{Prohorov metric} associated to our fixed metric $\rho$ on $X$ and given by
\begin{equation}\label{eqn:Prohorov}
D(\mu,\nu)=
\inf\{\eps>0: \mu(A)\le \nu(A^\eps)+\eps\text{ for all Borel sets }A\subset X\}, 
\end{equation}
where $\mu,\nu\in\M(X)$ and $A^\eps$ denotes the $\eps$-neighborhood of $A$, that is,
\[A^\eps=\{y\in X:\rho(x,y)<\eps\text{ for some }x\in A\}.\]
Although \eqref{eqn:Prohorov} does not seem to yield a symmetric function, we do obtain a metric since we consider only Borel \emph{probability} measures.

Given a Polish dynamical system $(X,T)$, we say that $\mu \in \M(X)$ is
\emph{$T$-invariant} if $\mu(T^{-1}(A))=\mu(A)$ for every Borel set $A \subseteq X$.
We let $\MT(X)$ be the set of $T$-invariant Borel probability measures on $X$.
If $X$ is compact, then $\MT(X)$ must be nonempty, but this is no longer true if $X$ is just Polish, even if $T$ exhibits non-trivial recurrence (see \cite[p. 374]{GM}).
A measure $\mu\in\MT(X)$ is \emph{ergodic} if  for every Borel set $A\subseteq X$ satisfying $T^{-1}(A)\subseteq A$ it holds that $\mu(A)$ is either $0$ or $1$. We write $\MTe(X)$ for the set of $T$-invariant ergodic measures. The set $\MT(X)$ is always convex, so $\MTe(X)$ is always non-empty if $\MT(X)\neq\emptyset$.

\begin{defn}
We say that $x \in X$ is a \emph{generic point} for
$\mu\in \MT(X)$ if the sequence of \emph{empirical measures}, whose $n$-th entry is given by
\begin{equation}\label{eqn:empirical}
 \Emp(x,n)= \frac{1}{n}(\delta_{T^0(x)}+\delta_{T^1(x)}+\dots +\delta_{T^{n-1}(x)})\end{equation}
converge in the \weakst\ topology on $\M(X)$ to $\mu$ as $n\to\infty$.
Here $\delta_x$ is the Dirac measure concentrated at $x\in X$.
We let $G_\mu=G_\mu(X,T)$ denote the set of generic points for $\mu$.
\end{defn}
If $\mu$ is ergodic, then it follows from the ergodic theorem that
$\mu$-almost every $x \in X$ is a generic point for $\mu$. We will not, however,
be assuming $\mu$ is necessarily ergodic in our results, and so the
existence of generic points is not automatic.

The notion of generic point can be viewed as generalizing the notion of a normal number.
If we take $X=\base^\omega$, $T$ the shift map on $\base^\omega$, and $\mu$ the standard Bernoulli
$(\frac{1}{\base},\ldots,\frac{1}{\base})$ measure on $\base^\omega$, then $x \in \base^\omega$ is a generic point
if and only if $x$ is the base $\base$ expansion of a normal number.

\section{Specification property and its weakening} \label{sec:spec}
In order to present our modification of the specification property we develop a terminology allowing us to compare various specification-like notions. It is inspired by \cite{KLO}.
Given a Polish dynamical system $(X,T)$ and $x\in X$, we define the \emph{orbit} of $x$ to be the set $O(x)=\{T^n(x):n\in\omega\}$.
Although the orbit is a set, it is customary to consider it as a sequence
$\langle T^n(x):n\in\omega\rangle$. Using this interpretation, and given $a,b\in\omega$ with $a\le b$, we define an \emph{orbit segment} of $x$ from $a$ to $b$ (synonymously, over $[a,b]$) to be the sequence $T^{[a,b]}(x)=\langle
T^a(x),\ldots,T^b(x)\rangle$. Every orbit segment (up to the choice of
$a$ and $b$) is determined by its \emph{initial point} $y=T^a(x)$ and
\emph{length} $n=b-a+1>0$, and hence without loss of generality an orbit segment may
be identified with a pair $\langle y,n\rangle$ in $X\times\omega\setminus\{0\}$. We will use these two formulations concurrently.
\begin{defn}
A  \emph{specification} of \emph{rank} $k>0$ is a finite sequence $\xi$ of $k$
orbit segments, that is, $\xi\in(X\times \omega\setminus\{0\})^k$. Let $\eps>0$. We say
that a point $y\in X$ \emph{$\eps$-close shadows} (or \emph{traces} or
\emph{follows})  a specification $\xi=\langle\langle
x_1,n_1\rangle,\ldots,\langle x_k,n_k\rangle\rangle\in(X\times\omega\setminus\{0\})^k$
if there are integers (called \emph{gaps}) $s_1,\ldots,s_{k-1}\in\omega$ such that
for every $1\le j\le k$ we have
\[
\rho( T^{\sum_{i=1}^{j-1} (n_i+s_i)+t}(y),T^t(x_j))<\eps\quad\text{for }0\le t<n_j.
\]
The number  $L=n_k+\sum_{i=1}^{k-1} (n_i+s_i)$ is the
\emph{number of iterates of $y$ required to trace $\xi$.} We may also say that
the \emph{orbit segment $\langle y,L\rangle$ traces $\eps$-close the specification $\xi$}.
\end{defn}

\begin{defn}
We say that $T$ has the \emph{specification property} if for every
$\eps>0$ there is a constant $N_\eps\in\omega$ such that for each $k>0$
and every specification $\xi\in(X\times\omega\setminus\{0\})^k$ there is a point $y$ which
traces it $\eps$-close with all gaps of size $\leq N_\eps$.
\end{defn}

To state our weakening of the specification property we will need the following notation:
By $\rho_\Lambda(x,y)$ we denote the Bowen distance  between $x,y\in X$ along a finite set
$\Lambda\subseteq \omega$ given by
\[
\rho_\Lambda(x,y)=\max_{j\in\Lambda} \rho(T^j(x),T^j(y)).
\]
The Bowen ball of radius $\eps>0$, along $\Lambda$ and centered at $x\in X$ is defined as
\[
B_\Lambda(x,\eps)=\{y\in X:\rho_\Lambda(x,y)<\eps\}.
\]
When $\Lambda\subseteq \{0,1,\ldots,n-1\}$ for some $n\in\N$ and
$|\Lambda|/n$ is close to $1$, then the Bowen ball $B_\Lambda(x,\eps)$
can be seen as a set of those points $y$ in $X$ whose orbits almost
$\eps$-trace the orbit segment $\langle x,n\rangle$. Here `almost'
means that the number of `errors' that is, those $0\le j<n$ for which
$\rho( T^j(x),T^j(y))\ge \eps$ is small. If
$\Lambda=\{0,1,\ldots,n-1\}$ for some $n\in\N$ then no errors are
allowed and we denote a Bowen ball $B_\Lambda(x,\eps)$ as
$B_n(x,\eps)$.


We will now introduce a weakening of the specification property
for a dynamical system.
Although a huge variety of generalizations of the specification
property are available in the literature \cite{KLO}, we find it convenient to
introduce yet another one. The newly defined property suffices to prove our main results.

\begin{defn}\label{def:faithful2}
We say that a point $y\in X$ \emph{$(\eps,\delta_1,\delta_2)$-approximately-traces a specification $\xi=(\langle x_j,n_j\rangle)_{j=1}^k$} if
\begin{enumerate}
\item $y\in B_{n_1}(x_1,\eps)$,
\item for each $2\le j \le k$ there is a set $\Lambda_{j}\subseteq \{0,\ldots,n_{j}-1\}$
with $|\Lambda_{j}|\ge (1 - \delta_1) n_{j} $ and an
integer $0\le s_{j-1}\le \delta_2 n_{j}$ such that
\[
T^{(\sum_{i=1}^{j-1}s_i+n_i)}(y)\in B_{\Lambda_j}(x_j,\eps).
\]
\end{enumerate}
\end{defn}

\begin{defn} \label{def:saps}
A Polish dynamical system $(X, T )$ has the \emph{strong approximate product
structure} if the following holds: Given $\eps, \delta_1, \delta_2 > 0$,
there exists a nonegative integer $N=N(\eps,\delta_1,\delta_2)$
such that for any specification $\xi=(\langle x_j,n_j\rangle)_{j=1}^k$ with $n_j \ge N$ for $1 \le j \le k$,
there exists a point $y \in X$ that $(\eps,\delta_1,\delta_2)$-approximately-traces $\xi$.
\end{defn}

\begin{rem}\label{rem:N}
We will also say that a point $y$ $\eps$-approximately-traces a specification $\xi$, if it $(\eps, \eps, \eps)$-approximately-traces $\xi$ as in Definition~\ref{def:faithful2}. We define $N(\eps) = N(\eps, \eps, \eps)$, where $N(\cdot,\cdot,\cdot)$ is as in Definition~\ref{def:saps}. 
\end{rem}

Thus, strong approximate product structure implies that, for any $\eps, \delta_1,\delta_2>0$ and any
specification $\xi= \langle (x_1,n_1),$ $ \dots, (x_n,n_k)\rangle $
with all $n_i$ sufficiently large, we can find a point $y$
which $\eps$-traces the first orbit segment $(x_1,n_1)$,
and $\eps$-traces the segments $(x_i,n_i)$ for $i\ge 2$ with at most
$\delta_1 n_i$ errors, allowing gaps of sizes at most $\delta_2 n_i$
between the segments $(x_{i-1}, n_{i-1})$ and $(x_i,n_i)$.

The strong approximate product structure is similar to and implies the \emph{approximate product structure} of Pfister and
Sullivan \cite{PS}. The difference between these two notions is that in Definition~\ref{def:faithful2},
 $\eps$-approximate-tracing allows no errors in the first orbit segment. Below we picture the implications between the strong approximate product structure (SAPS) and specification properties investigated by other authors: the specification property (SP), the weak specification property (WSP), and approximate product structure (APS), see \cite{KLO} for more details:
\begin{gather*}
\textrm{SP}  \Rightarrow
\textrm{WSP} \Rightarrow
\textrm{SAPS} \Rightarrow
\textrm{APS}
\end{gather*}

Although definition of the specification property looks prohibitively restrictive, there are many classes of systems exhibiting a variant of that property.
Familiar examples of symbolic systems (subshifts) that have the specification property are mixing shifts of
finite type and mixing sofic shifts. There are many other symbolic systems with the specification property. All topologically mixing graph maps (in particular, interval maps) have the specification property.
The specification property is closely related with hyperbolicity: every transitive uniformly hyperbolic diffeomorphism of a compact connected manifold has the specification property.
Specification or its variants also holds for systems of algebraic origin. Following Lind \cite{Lind}, we say that a toral automorphism is quasi-hyperbolic if
its associated matrix $A$ has no roots of unity as eigenvalues.
Lind \cite[Theorem (ii)--(iii)]{Lind} and Marcus \cite{Marcus} proved that all quasi-hyperbolic toral automorphisms must satisfy weak specification.
As  the strong approximate product structure implies  the feeble specification property used in \cite{AJKM}, our results also apply to the systems considered in \cite{AJKM}, most notably to $\beta$-shifts.

\section{Sets defined by conditions on statistical behaviour of their orbits}\label{sec:SOI-def}

Generalizing the notion of a generic point, we define $V(x)$ as the set of all limit points in $\M(X)$
of the sequence $(\Emp(x,n))_{n\ge 0}$. Thus if $x\in G_\mu$ then $V(x)=\{\mu\}$, and the converse is true if $X$ is compact.
It is well known that $V(x)$ is a closed connected subset of $\MT(X)$,
which is necessarily nonempty provided that $X$ is compact.

Given a set $V\subseteq \MT(X)$, we can relate it to $V(x)$ and define sets of points, whose orbits exhibit particular statistical behaviour.

\begin{defn}
If $V\subseteq \M(X)$, we let $G^V=G^V(X,T)=\{ x \in X \colon V\subseteq V(x) \}$.
Similarly, let $G_V=G_V(X,T)=\{ x \in X \colon V(x) \subseteq V\}$.
We set  $G(V)= G^V \cap G_V$. Finally, we let $\prescript{V}{} G =
\{ x \in X \colon V(x)\cap V \neq \emptyset\}$.
\end{defn}

We refer to the set $G(V)$ as the set of \emph{generic
points for $V$}. We caution the reader that in the non-compact case
it is not necessarily true that $G_\mu=G(\{ \mu\})$ (see Remark~\ref{nts}).

\begin{defn}
Given $\phi\in\mathcal{C}_b(X)$, $k>0$ and $y\in X$ we define \emph{$k$-th ergodic average of $\phi$ along the orbit of $y$} by
\begin{equation}\label{eqn:Ak}
A_k\phi(y)=\frac{1}{k}\sum_{j=0}^{k-1} \phi(T^j(y))=\int \phi\, d \! \Emp(y,k).
\end{equation}
\end{defn}
\begin{defn}
Given $\phi\in\mathcal{C}_b(X)$
we say that a point $x\in X$ is \emph{Birkhoff regular} for $\phi$
if the sequence of ergodic averages $A_k\phi(x)$ converges as $k\to\infty$. For $\alpha\in\R$
we define the \emph{$\alpha$-level sets for Birkhoff averages} (\emph{$\alpha$-Birkoff regular points}) of $\phi$ by
\[
B_{\phi}(\alpha)=\{x\in X: \lim_{k\to\infty} A_k\phi(x)=\alpha\}.
\]
The set of all Birkhoff regular points for $\phi$ is denoted
\[
B_\phi(X,T)=\bigcup_{\alpha\in\R}B_\phi(\alpha).
\]
\end{defn}
\begin{defn}
The \emph{irregular set for $\phi$} is the set $I_\phi(X,T)$ consisting
of all points $y\in X$ such that the sequence of ergodic averages
$(A_k\phi(y))_{k\in\N}$ does not converge.
The \emph{irregular set} $I(X, T)$
of the Polish dynamical system $(X,T)$ is the union of $I_\phi(X,T)$
when $\phi$ runs over all bounded continuous real-valued functions on
$X$.
\end{defn}

Thus, for each $\phi\in\cC_b(X)$ we have a partition
\[
X=I_\phi(X,T)\cup \bigcup_{\alpha\in\R}B_\phi(\alpha).
\]
By the Birkhoff ergodic theorem,
the irregular set is
universally null:
$\mu(I_\phi(X,T))=\mu(I(X, T))=0$ for every
$T$-invariant Borel probability measure $\mu$ on $X$ and $\phi\in C_b(X)$.  The
complement of the irregular set, denoted $Q(X,T)$, is called the
\emph{quasi-regular set} of $(X,T)$.
If $X$ is compact, then it is easy to see that every quasi-regular point must be also generic for some $T$-invariant measure (cf. \cite{Oxtoby}).
It seems to be a little known fact that each quasi-regular point must be generic for some $T$-invariant measure also in the general situation, that is, for every Polish metric space.
For example, Oxtoby in \cite[Section 7]{Oxtoby} (following Fomin) adds an additional condition to the definition of quasi-regular point for a Borel automorphism to get a similar result.

\begin{fact}\label{fact:Q=G}
We have $Q(X,T)=\bigcup_{\mu \in \MT(X)} G_\mu$.
\end{fact}
\begin{proof}
Note that if $x$ is generic for $\mu$, then the definition of \weakst \! convergence and \eqref{eqn:Ak} imply that
$A_k\phi(y)\to\int\phi\, d \mu$ as $k\to\infty$, so every generic point is $\alpha(\phi)$-Birkhoff regular for
every $\phi\in\mathcal{C}_b(X)$ with $\alpha(\phi)=\int\phi\, d \mu$. In particular, each generic point is quasi-regular.
On the other hand, the space of Borel measures on $X$ is weakly sequentially complete by \cite[Theorem 8.7.1]{Bogachev}.
Here, weak sequential completeness means that every fundamental sequence converges (cf. \cite[Section 8.1]{Bogachev}). A sequence $(\mu_n)_{n\ge 1}$ of Borel probability measures is fundamental if for every bounded continuous function $f$, the sequence $\int f \text{d}\mu_n$ is Cauchy. Therefore if $x$ is quasi-regular, then the sequence $(\Emp(x,n))_{n\ge 1}$ is fundamental and so converges to a Borel probability measure $\mu$ on $X$, which is easily seen to be $T$-invariant.
\end{proof}


Note that for all Polish dynamical systems $(X,T)$,
for each $\mu\in\MT(X)$, and for each $\phi\in\mathcal{C}_b(X)$ we have
\[
G_\mu\subseteq Q(X,T) \subseteq B_\phi(X,T)
\quad\text{and}
\quad
I_\phi(X,T)\subseteq I(X,T).
\]

\section{Upper bounds} \label{sec:ub}
We will study which sets defined by placing conditions on the statistical behaviour of orbits are Borel
(assuming that $X$ is compact if necessary). Recall that $(X,T)$ is a Polish dynamical system and
$(f_n)$ is a sequence of bounded continuous real-valued functions on $X$
which generate the \weakst\ topology on $\M(X)$ (cf.\ Theorem 17.19 of \cite{Kechris}).
That is, $\mu_k \to \mu$ as $k\to\infty$ if and only if for each $n$ we have $\int f_n\, d\mu_k \to \int f_n\, d\mu$ as $k\to\infty$.
Recall also that $A_k\phi$ stands for the $k$-th ergodic average of $\phi$ along the orbit of $y$, see \eqref{eqn:Ak}.

\begin{fact} \label{upper:Bfi}
If $\phi\in\mathcal{C}_b(X)$ and $\alpha\in\R$, then the set
$B_\phi(\alpha)=\{x\in X:\lim_{k\to\infty}A_k\phi(x)=\alpha\}$ of $\alpha$-Birkhoff regular points is a $\bP^0_3$ set.
\end{fact}
\begin{proof}
It is easy to see that
\[
B_\phi(\alpha)=\bigcap_{m=1}^\infty\bigcup_{K=1}^\infty\bigcap_{k=K}^\infty
\{y\in X\colon | A_k\phi(y)-\alpha|\le 1/m\}. \qedhere
\]
\end{proof}

\begin{fact} \label{upper:Gmu}
If $\mu\in\MT(X)$, then the set $G_\mu$
of $\mu$-generic points is a $\bP^0_3$ set.
\end{fact}

\begin{proof}
A point $y\in X$ belongs to $G_\mu$ if and only if for every $n$ the sequence
$A_k f_n(y)$ converges to $c_n=\int f_n\,d\mu$ as $k\to\infty$. Thus
\[
G_\mu=\bigcap_{n=0}^\infty B_{f_n}(c_n).
\]
By Fact \ref{upper:Bfi},  $G_\mu$ is a $\bP^0_3$ set.
\end{proof}

\begin{fact} \label{upper:irregular-phi}
If $\phi\in\mathcal{C}_b(X)$, then the irregular set $I_\phi(X,T)$ is a $\bS^0_3$ set
and $B_\phi(X,T)=X\setminus I_\phi(X,T)$ is a $\bP^0_3$ set.
\end{fact}
\begin{proof} A point $x\in X$ is irregular for $\phi$ if and only if 
\[
\liminf_{k\to\infty}A_k\phi(x)<\limsup_{k\to\infty}A_k\phi(x).
\]
Therefore
\[
I_\phi(X,T)=\bigcup_{N=1}^\infty \bigcap_{K=1}^\infty \bigcup_{k= K}^\infty
\bigcup_{\ell= K}^\infty \{x\in X\colon A_k\phi(x)+1/N<A_\ell \phi(x)\}.
\]
By continuity, the set $\{x\in X\colon A_k\phi(x)+1/N<A_\ell \phi(x)\}$ is open in $X$ for any choice of $\phi$ and $k,\ell,m$. It follows that
$I_\phi(X,T)$ is a $\bS^0_3$ set.
\end{proof}

\begin{fact} \label{upper:irregular}
The irregular set $I(X,T)$ is a $\bS^0_3$ set and the
quasi-regular set $Q(X,T)=X\setminus I(X,T)$ is a $\bP^0_3$ set.
\end{fact}
\begin{proof}
Note that by Fact \ref{fact:Q=G} we have $x\in Q(X,T)$ if and only if $x$ is generic for some $\mu\in\MT(X)$. The latter is equivalent to
the sequence $(\Emp(x,n))_{n\ge 1}$ being a Cauchy sequence in the Prohorov metric $D$ (starting with the complete metric $\rho$ on $X$ we obtain a complete metric $D$ on $\M(X)$, see \cite[Cor. 18.6.3]{Garling}). Hence
\[
Q(X,T)=\bigcap_{m=1}^\infty\bigcup_{N=1}^\infty \bigcap_{k= N}^\infty
\bigcap_{\ell= N}^\infty \{ x \in X \colon D(\Emp(x, k), \Emp(x,\ell))  \le 1/m \}.
\]
Note that for any $n \geq 1$, the map $x \mapsto \Emp(x, n)$ is continuous. Thus the set $\{ x \in X \colon D(\Emp(x, k), \Emp(x,\ell))  \le 1/m \}$ is closed for any choice of the parameters $k,\ell$ and $m$. Therefore $Q(X,T)$ is a $\bP^0_3$ set.
\end{proof}


\begin{fact} \label{upper:GW}
If $V\subseteq \MT(X)$ is closed, then
$G^V=\{x\in X\colon  V\subseteq V(x)\}$ is a $\bP^0_2$ set.
\end{fact}

\begin{proof}
Let $V_0$ be a countable dense subset of $V$. Then
\[
G^V = \bigcap_{\nu \in V_0} G^{\{ \nu \} } =
\bigcap_{\nu \in V_0}
\bigcap_{N = 1}^\infty
\bigcap_{m = 1}^\infty
\bigcup_{n = N}^\infty
\{ x \in X \colon D(\Emp(x, n), \nu)  < 1/m \}.
\]
Note that for any $n \geq 1$, the map $x \mapsto \Emp(x, n)$ is continuous. Thus the set $\{ x \in X \colon D(\Emp(x, n), \nu)  < 1/m \}$ is open for any choice of the parameters $n$ and $\nu$. Thus $G^V$ is a $\bP^0_2$ set.
\end{proof}

Recall that a point $x \in X$ is said to have \emph{maximal oscillation}, if $V(x) = \MT(X)$ \cite[Def. 21.17]{DGS}.
Taking $V = \MT(X)$ in Fact \ref{upper:GW}, we immediately obtain the following corollary:

\begin{fact}
The set of points with maximal oscillation is a $\bP^0_2$ set.
\end{fact}

%
%
%

\begin{fact}\label{fact:UG-analytic}If $U\subseteq \MT(X)$ is open, then the set
$\prescript{U}{}G$ is analytic (is $\bs^1_1$).
\end{fact}
\begin{proof}
Let $R=\{(x,\mu)\in X\times \M(X):\mu\in V(x)\}\subseteq X\times\MT(X)$. We have
\[
R=\bigcap_{m=1}^\infty \bigcap_{N=1}^\infty \bigcup_{n=N}^\infty \{(x,\mu) \in X\times \MT(X) :D(\mu,\Emp(x,n))<1/m\}.
\]
Clearly, the set $\{(x,\mu) \in X\times \MT(X) :D(\mu,\Emp(x,n))<1/m\}$ is open in $X\times\MT(X)$ for every $n,m>0$, so $\{(x,\mu)\in X\times \M(X):\mu\in V(x)\}\subseteq X\times\MT(X)$ is Borel.
Since $\prescript{U}{}G$ is a projection of the set $R\cap (X\times U)$ 
onto the first coordinate, we get that $\prescript{U}{}G$ is $\bs^1_1$ (see \cite[p. 86]{Kechris}).
\end{proof}
As an immediate corollary we have:

\begin{fact}\label{GV_co-analytic}For every closed $V\subseteq \MT(X)$ the set $G_V$ is co-analytic (is $\bp^1_1$).
\end{fact}
\begin{proof}
Recall that  $\prescript{X \setminus V}{} G = X \setminus  G_{V} $ and apply Fact~\ref{fact:UG-analytic}.
\end{proof}

\begin{fact} \label{upper:UG}
Assume $X$ is compact. If $V\subseteq \MT(X)$ is closed, then
$G_V=\{x\in X\colon V(x)\subseteq V\}$ is a $\bP^0_3$ set.
If $U \subseteq \MT(X)$ is open, then $\prescript{U}{} G = \{x\in X: V(x)\cap U\neq\emptyset\}$ is a $\bS^0_3$ set.
\end{fact}

\begin{proof}
Observe that
\begin{align*}
V(x) \subseteq V \Leftrightarrow & \;
D(\Emp(x, n), V) \to 0 \textrm{ as } n \to \infty \\
\Leftrightarrow & \;
x \in \bigcap_{m = 1}^\infty
\bigcup_{N = 1}^\infty
\bigcap_{n = N}^\infty
\{ x \in X : D(\Emp(x, n), V) \le 1/m \}.
\end{align*}
Here we employed the standard notation $D(\Emp(x, n), V) = \inf_{v \in V} D(\Emp(x, n), v)$. We use the compactness of $X$ in the first equivalence; without compactness, only the implication~$\Leftarrow$~holds.  Since the map $x \mapsto \Emp(x, n)$ is continuous for every $n \geq 1$, the sets $\{ x \in X : D(\Emp(x, n), V) \le 1/m \}$ are closed. Thus $G_V$ is a $\bP^0_3$ set. For the second part of the fact, note that $\prescript{U}{} G = X \setminus  G_{X \setminus U}$ for every open set $U\subseteq X$.
\end{proof}

\begin{fact} \label{upper:GV}
If $V\subseteq \MT(X)$ is closed, then  the set $G(V)$ is co-analytic (is $\bp^1_1$). If, in addition, $X$ is
 compact, then $G(V)$ is a $\bP^0_3$ set.
\end{fact}

\begin{proof}
Use the equality $G(V)=G^V \cap G_V$, Fact~\ref{upper:GW} and, respectively, Fact \ref{GV_co-analytic} or Fact \ref{upper:UG}.
\end{proof}

%
%



The following question is related to those considered in this paper.

\begin{ques*} \label{ques:gq}
Given a dynamical system $(X,T)$, $\varphi\in \cC_b(X)$, and
$A\subseteq \R$ what is the exact complexity of the set
\[
B(\varphi,A)=\bigcup_{\alpha\in A}B_\phi(\alpha)?
\]
\end{ques*}

\section{Auxiliary results}\label{sec:aux}
The next three results record some basic estimates which we will use in the proof of our main result.
To formulate these results we have to extend our terminology.

The {\em empirical measure} $\Emp(\bar x)$ of a sequence $\bar x=\langle x_1,\ldots,x_k\rangle\in X^k$, where $k>0$
is the normalized sum of the Dirac (point-mass) measures along points in the sequence, that is
\[
\Emp(\bar x)=\frac{1}{k}\sum_{p=1}^{k}\delta_{x_p}.
\]
Given a specification $\xi=(\langle x_j,n_j\rangle)_{j=1}^k\in (X\times \omega\setminus\{0\})^k$ we define the \emph{empirical measure of $\xi$} denoted by $\Emp(\xi)$ as the empirical measure of the sequence $\seq(\xi)\in X^{\sum_{j=1}^kn_j}$ obtained by concatenation of the orbits segments in $\xi$, that is,
$\Emp(\xi)=\Emp(\seq(\xi))$  where
\[
\seq(\xi)=\langle x_1,T(x_1),\ldots, T^{n_1-1}(x_1),\ldots,x_k,T(x_k),\ldots,T^{n_k-1}(x_k)\rangle\in X^{\sum_{j=1}^kn_j}.
\]
Note that if $\xi=\langle (x,n)\rangle$ is a single orbit segment for some $n>0$, then we have
\[
\Emp(\xi)=\Emp(\langle x,T(x),\ldots, T^{n-1}(x)\rangle)=
\frac{1}{n}\sum_{j=0}^{n-1}\delta_{T^j(x)}=\Emp(x,n),
\]
where $\Emp(x,n)$ is the empirical measure defined by \eqref{eqn:empirical}.

\begin{lem}\label{lem:close-segments}
Let $\gamma,\delta,\eps>0$. Assume that
$\bar x=\langle x_1,\ldots,x_k\rangle\in X^k$ and
$\bar y=\langle y_1,\ldots,y_\ell\rangle\in X^\ell$, where $0<k\le \ell\le (1+\delta)k$.
If there exists $j$ with $(1-\gamma)k\le j\le k$ and
subsequences $\bar x'=\langle x_{a_1},\ldots,x_{a_j}\rangle\in X^j$ and
$\bar y'=\langle y_{b_1},\ldots,y_{b_j}\rangle\in X^j$ with $1\le a_1 <\ldots < a_j\le k$
and $1\le b_1 <\ldots < b_j\le \ell$
such that $\rho(x_{a_i},y_{b_i})<\eps$ for $1\le i\le j$, then
\[
D(\Emp(\bar x),\Emp(\bar y))
\le 2\gamma+\delta+\eps.
\]
\end{lem}

\begin{proof}
By the triangle inequality,
\begin{align} \label{triangle_ineq}
D(\Emp(\bar{x}),\Emp(\bar{y})) \leq D((\Emp(\bar{x}), \Emp(\bar x'))
+D(\Emp(\bar x'), \Emp(\bar y'))+D(\Emp(\bar y'), \Emp(\bar{y})).
\end{align}
For each $1\leq i \leq j$ we have from the definition of $D$ that
$D(\delta_{x_{a_i}},\delta_{y_{b_i}})<\epsilon$. It follows that
$D(\Emp(\bar x'), \Emp(\bar y'))<\epsilon$.

Next, observe that
$$ \Emp(\bar{x}')(A) \le \frac{k}{j} \Emp(\bar{x})(A)
 \le \Emp(\bar{x})(A) + \frac{k-j}{k}
 \le \Emp(\bar{x})(A) + \gamma,
$$
which means that $D(\Emp(\bar{x}), \Emp(\bar x'))\leq \gamma$. By our assumptions $(1 - \gamma - \delta) k \le l \le k$; this implies $D(\Emp(\bar{y}), \Emp(\bar y')) \leq \gamma +  \delta$ by a similar reasoning. Plugging all the mentioned inequalities into (\ref{triangle_ineq}) finishes the proof.
\end{proof}


We have the following immediate corollary.

\begin{cor}Let $k>0$.
If a point $y$ $\eps$-approximately-traces a specification
$\xi=(\langle x_j,n_j\rangle)_{j=1}^k\in (X\times \omega\setminus\{0\})^k$
with gaps $s_1,\ldots,s_{k-1}\in \omega$ bounded by $\delta_2 n_j$
and errors bounded by $\delta_1 n_j$ and
\[
L=n_k+\sum_{i=1}^{k-1}(s_i+n_i),
\]
then $D(\Emp(\xi),\Emp(y,L))<2\gamma+\delta+\eps$.
\end{cor}

\begin{lem}\label{lem:protogen}
If $(X,T)$ is a Polish dynamical system with the strong approximate
product structure and $\mu$ is a $T$-invariant measure on $X$, then
for each $\eps>0$ there exists a positive integer $Q$ such that for every $n>0$
there is $x\in X$ satisfying $D(\mu,\Emp(x,nQ))<\eps$.
\end{lem}

\begin{proof} Assume that $\mu$ is a $T$-invariant (possibly non-ergodic) probability measure.
It is well-known that $\mu$ can be approximated by  affine combinations of finite collections of
ergodic measures with rational coefficients, that is, for every $\eps>0$ there is a measure $\nu\in\MT(X)$ such that
\[
D(\mu,\nu)<\eps/4\quad \text{and}\quad\nu=\sum_{i=1}^{k} q_i \nu_i,
\]
where $k>0$ and $\nu_1,\ldots,\nu_k\in\MT(X)$ are ergodic, $q_1,\ldots,q_k\in\mathbb{Q}$ are positive
and $q_1+\cdots+q_k=1$. Let $Q$ be a positive integer such that for some $p_1,\ldots,p_k$ we have $q_i=p_i/Q$ for $i=1,\ldots,k$.
We choose $z_1,\ldots,z_k$ such that $z_i$ is generic for $\nu_i$.
Next, let $K\geq N(\eps/8)$ (here $N(\cdot)$ is as in Remark~\ref{rem:N}) be large enough so that
\[
D(\Emp(z_i,K),\nu_i)<\eps/8 \text{ for }i=1,\ldots,k.
\]
For each $n>0$ we define a specification $\xi$ by
\[
\xi=(\langle x'_j,K\rangle)_{j=1}^{nQ}\in (X\times \omega\setminus\{0\})^{nQ},
\]
where each $z_i$ appears $np_i$ times on the list $x'_1,\ldots,x'_{nQ}$, that is,
\[
\langle x'_j\rangle_{j=1}^{nQ}=\langle \underbrace{z_1,\ldots,z_1}_{\text{$np_1$ times}},
\underbrace{z_2,\ldots,z_2}_{\text{$np_2$ times}},\,\ldots\,,\underbrace{z_k,\ldots,z_k}_{\text{$np_k$ times}}\rangle.
\]
By definition of $\xi$ we have
\begin{equation}\label{eqn:emp-xi}
\Emp(\xi)=\sum_{i=1}^k q_i\Emp(z_i,K) .
\end{equation}
Let $x$ be a point which $\eps/8$-approximately-traces the specification $\xi$ and let $L$ be
the number of iterates required to do so. Then
$D(\Emp(\xi),\Emp(x,L))<\eps/4$, hence $D(\mu,\Emp(x,L))<5\eps/8$
because of \eqref{eqn:emp-xi}.

Since $nQ\le L<(1+\eps/4)nQ$ we use Lemma~\ref{lem:close-segments} to get
$D(\Emp(x,nQ),\Emp(x,L))<\eps/4$ which ends the proof.
\end{proof}

\begin{rem} \label{rem:wsl}
Lemma~\ref{lem:protogen} shows, assuming $(X,T)$ has the strong approximate product structure, that
the set of measures of the form $\Emp(x,n)$ for $x\in X$ and $n>0$
is weak${}^*$ dense in $\MT(X)$ in the sense that every $\mu \in \MT(X)$ is the
weak${}^*$ limit of a sequence $\mu_k=\Emp(x_k,n_k)$ (in general, the measures $\Emp(x,n)$
are not $T$-invariant).
\end{rem}

\section{Main Results} \label{sec:ws}

In this section we prove our main results about the complexity of
the set of generic points for systems satisfying the strong approximate
product structure. If the system $(X, T)$ is compact, the set $\MT(X)$ is nonempty (a theorem of Bogolyubov and Krylov). If $X$ is compact and $\MT(X)$ consists of a single element $ \mu $, then $X = G_\mu$. Thus for compact $X$, the statistical behavior of the system is trivial if there exist less than two $T$-invariant measures. For this reason, in this section we mostly work under the assumption that $| \MT(X) | \ge 2$. It should be kept in mind, however, that for general Polish $X$, the set $\MT(X)$ may be empty.

\begin{exam} Let $\hat X$ be a compact metric space and $\hat{T}\colon \hat{X}\to\hat{X}$ be a continuous map with the specification property. For example, we can take the compact unit interval, $\hat{X}=[0,1]$ and $\hat{T}$ to be the tent map $\hat{T}(x)=1-|1-2x|$. Let $X$ be the set of points with the maximal oscillation. It was noted by Sigmund that  $X$ is a dense $G_\delta$ subset of $\hat{X}$, and $X$ is a $T$-invariant set. Hence restricting $\hat{T}$ to $X$ we obtain a Polish dynamical system which does not have any invariant measures, since $I(X,T)=X$. It is easy to see that $(X,T)$ has the specification property. It shows that we need to assume that $| \MT(X) | \ge 2$
\end{exam}

In Lemma \ref{lem:main}, we show (assuming that $(X,T)$ has the strong approximate product structure) that given a sequence of $T$-invariant measures $(\mu_n)_{n \ge 1}$, satisfying $D(\mu_n, \mu_{n+1}) \to 0$, one can find a point $y \in X$ and piecewise constant\footnote{Constant on finite sets of consecutive $j$'s.} nondecreasing function $j\mapsto n(j)$ such that the sequence of empirical measures $\Emp(y, j)$, is asymptotic to the sequence $\mu_{n(j)}$ (the distance
$D(\Emp(y,j),\mu_{n(j)})$ converges to $0$ as $j\to\infty$).

Part~(\ref{main1}) of Theorem~\ref{thm:main} gives the exact complexity of the set of generic
points $G_\mu$ for a Polish dynamical system $(X,T)$ with the strong approximate product structure satisfying $| \MT(X) | \ge 2$.
More generally,
Part~(\ref{main2}) of Theorem~\ref{thm:main}, uses the same proof method, to obtain lower bounds for the complexity of the set $G(V)$, where $V\subseteq \MT(X)$is  a non-trivial, closed,
connected subset of $\MT(X)$. If $X$ is compact then $G_\mu=G(\{ \mu\})$ and
so in this case (\ref{main1})  follows directly from the statement of
(\ref{main2}) of Theorem~\ref{thm:main}. In this section we also prove lower bounds on the complexity of the sets $G^V, G_V$ and other sets introduced in Section \ref{sec:ub}.

\begin{lem} \label{lem:main}
Let $(X, T)$ be a Polish dynamical system with the strong approximate product structure.
\begin{enumerate}
\item \label{it:i} Given a sequence $(\mu_n)_{n \geq 1}$ in $\MT(X)$ such that $D(\mu_n, \mu_{n+1}) \to 0$, there exists a point $y \in X$ and a sequence $L_n \nearrow \infty$ such that
\begin{equation} \label{eqn:lemmain}
D(\Emp(y, j), \mu_{n(j)}) \to 0 \quad \textrm{ as } j \to \infty,
\end{equation}
where $n(j)$ denotes the unique integer satisfying $L_{n(j) - 1} \le j < L_{n(j)}$. In particular, this implies $V(y) = \{ \textrm{ limit points of the sequence } (\mu_k)_{k \geq 1} \}$.
\item \label{it:ii} Given any $y_0 \in X$ and $\eps > 0$, we can define $y$ in \eqref{it:i} so that $y \in B(y_0, \eps)$.
\item \label{it:iii} If $y^{1}$ and $y^{2}$ are points obtained as in \eqref{it:i}  for sequences of measures $( \mu^{1}_n ) _{n \ge 1}$ and $( \mu^{2}_n ) _{n \ge 1}$ such that for some $N\ge 1$ we have that $\mu^{1}_i = \mu^{2}_i$ for $1\le i \le N$,  then $\rho(y^{1}, y^{2}) \le 1 / 2^N$.
\end{enumerate}
\end{lem}

\begin{proof}
Fix $0< \eps < 1$.
We will define sequences of integers $(K_n)_{n\ge 0}$ and points
$(x_n)_{n\ge 0} \subseteq X$ which will be parameters for our construction. For each $n\ge 0$ we use Lemma \ref{lem:protogen}
to pick $x_n \in X$ and $K_n$
so that the following hold:
\begin{gather}
 K_n\ge N(\eps/2^{n+1}), \label{eqn:cond-x}\\
  D(\Emp(x_n,K_n),\mu_n) \le \eps/2^{n+1}.  \label{eqn:cond-i}
\end{gather}
Here, $N(\eps/2^{n+1})$ is defined as in Remark~\ref{rem:N}. 

Next, we will define $(\ell_n)_{n\ge 1}, (L_n)_{n \ge 0}$ and $(y_n)_{n \ge 0} \subseteq X$. We begin with $L_0 = K_1$ and $y_0$ arbitrary point in $X$. For all $n \ge 0$ we will have
\begin{equation} \label{eqn:Ln_bound}
L_{n-1} \ge n K_n.
\end{equation}
We proceed inductively, for $n \ge 1$:
\begin{itemize}
  \item pick $\ell_n$ such that $\ell_n K_n \ge (n + 1) K_{n+1}$ and $\ell_n K_n \ge n L_{n-1}$,
  \item define the specification $\xi_n$ to consist of the pair $(y_{n-1}, L_{n-1})$ followed by the pair $(x_n, K_n)$ repeated $\ell_n$ times,
  \item define $y_n$ to be a point that $\eps/2^{n+1}$-approximately-traces the specification $\xi_n$,
  \item define $L_n$ to be the length of orbit segment required for $y_n$ to trace $\xi_n$. Note that condition (\ref{eqn:Ln_bound}) is satisfied, as $L_n \ge \ell_n K_n \ge n K_{n+1}$.
\end{itemize}
It follows from this definition that
\begin{equation}
\rho(T^i(y_{n-1}), T^i(y_n)) \le \eps / 2^{n+1} \quad \textrm{ for } 0 \le i < L_{n-1}.
\end{equation}
In particular, the sequence $(y_n)$ must converge to a limit $y$, and we have
\begin{equation} \label{eqn:yn_close_y}
\rho(T^i(y_{n-1}), T^i(y)) \le \eps / 2^n \quad \textrm{ for } 0 \le i < L_{n-1}.
\end{equation}

We will now examine the empirical measures $\Emp(y, j)$ for $j>0$. First, consider $\Emp(y, L_n)$ for some integer $n$. By (\ref{eqn:yn_close_y}), we have
\begin{equation} \label{eqn:item1}
D(\Emp(y, L_n), \Emp(y_n, L_n)) \le \eps / 2^n.
\end{equation}
Recall that we have chosen $y_n$ so that it $\eps/2^{n+1}$-approximately-traces $\xi_n$ along the orbit segment of length $L_n$. Together with Lemma \ref{lem:close-segments} this gives $D(\Emp(y_n, L_n), \Emp(\xi_n)) < 4 \eps/2^{n+1} = \eps / 2^{n-1}$.

Define $\bar{\xi_n}$ to be the specification consisting of $\ell_n$ repetitions of pair $(x_n, K_n)$. Thus $\bar{\xi_n}$ only differs from $\xi_n$ by the lack of initial $(y_{n-1}, L_{n-1})$ segment. We defined $\ell_n$ so that $\ell_n K_n \ge n L_{n-1}$, thus an application of Lemma \ref{lem:close-segments} gives
\begin{equation} \label{eqn:item2}
D(\Emp(\xi_n), \Emp(\bar{\xi_n})) \le 1/n.
\end{equation}
Finally, notice that
\begin{equation} \label{eqn:item3}
D(\Emp(\bar{\xi_n}), \mu_n) = D(\Emp(x_n, K_n), \mu_n) \le \eps / 2^{n+1}.
\end{equation}

Combining (\ref{eqn:item1}), (\ref{eqn:item2}) and (\ref{eqn:item3}) gives
\begin{equation} \label{eqn:emp_measure_Ln}
D(\Emp(y, L_n), \mu_n) \le 1/n + 7 \eps/2^{n+1} \to 0 \quad \textrm{ as } n \to \infty.
\end{equation}

We are now prepared to estimate $\Emp(y, j)$ for arbitrary $j>0$. There exists a unique integer $n = n(j)$ such that $L_{n-1} \le j < L_n$. Recall the definition of $y_n$ as a point that $\eps / 2^{n+1}$-approximately-traces $\xi_n$. Let $s_1, \dots s_{\ell_n}$ be the lengths of gaps in this tracing, as in Definition \ref{def:faithful2}. The orbit segment $(y_n, L_n)$ is made up of orbit segments corresponding to the elements of $\xi_n$ (interspersed by gaps of length $s_i$). Thus, the initial subsegment $(y, j)$ fully contains some number $b+1$ of those segments. Formally, $b$ is the unique integer $0 \le b < \ell_n$ such that
\[
B := L_{n-1} + bK_n + \sum_{i < b} s_i \le j < L_{n-1} + (b+1)K_n + \sum_{i \le b} s_i.
\]
We have $j - B \le K_n + s_b \le (1 + \eps / 2^{n+1}) K_n$. We also have
\begin{equation} \label{eqn:B_ineq}
\frac{j - B}{B} \le \frac{(1 + \eps / 2^{n+1}) K_n}{B} \le
\frac{(1 + \eps / 2^{n+1}) K_n}{L_{n-1}} \le \frac{1 + \eps / 2^{n+1}}{n},
\end{equation}
where the last inequality follows from (\ref{eqn:Ln_bound}). Next, note that $y_n$, over the orbit segment of length $B$, $\eps / 2^{n+1}$-approximately-traces the specification
\[
\zeta = \langle
(y_{n-1}, L_{n-1}), \underbrace{(x_n, K_n), \dots (x_n, K_n)}_{b \textrm{ times}}
\rangle.
\]
By the triangle inequality, we obtain
\begin{multline} \label{eqn:temp12}
D(\Emp(y, j), \mu_{n}) \le
D(\Emp(y, j), \Emp(y, B)) + D(\Emp(y, B), \Emp(y_n, B)) + \\
D(\Emp(y_n, B), \Emp(\zeta)) + D(\Emp(\zeta), \mu_{n}).
\end{multline}
Let us examine the terms on the right hand side of $\le$ in \eqref{eqn:temp12}:
\begin{itemize}
  \item $D(\Emp(y, j), \Emp(y, B)) \le (1 + \eps / 2^{n+1}) / n$ by (\ref{eqn:B_ineq}),
  \item $D(\Emp(y, B), \Emp(y_n, B)) \le \eps / 2^{n+1}$ by (\ref{eqn:yn_close_y}),
  \item $D(\Emp(y_n, B), \Emp(\zeta)) \le 4\eps / 2^{n+1}$ by $\eps / 2^{n+1}$-approximate-tracing and Lemma \ref{lem:close-segments},
  \item $D(\Emp(\zeta), \mu_{n}) \le \max \{ D(\Emp(y_{n-1}, L_{n-1}), \mu_{n}), D(\Emp(x_n, K_n), \mu_{n}) \}$.
  Now,
  \begin{itemize}
    \item[$\circ$] $D(\Emp(y_{n-1}, L_{n-1}), \mu_{n-1}) \to 0$ as $n \to \infty$, as shown by (\ref{eqn:emp_measure_Ln}), and by our assumption $D(\mu_n, \mu_{n-1}) \to 0$,
    \item[$\circ$] $D(\Emp(x_n, K_n), \mu_{n}) \le \eps / 2^{n+1} \to 0$ as $n \to \infty$.
  \end{itemize}
\end{itemize}
All of this, applied to the right hand side of $\le$ in \eqref{eqn:temp12}, shows that
\begin{equation}
D(\Emp(y, j), \mu_{n(j)}) \to 0 \quad \textrm{ as } j \to \infty.
\end{equation}
It follows easily that $V(y) \subseteq \{ \textrm{ limit points of the sequence } (\mu_k)_{k \geq 1} \}$. Indeed, if $\Emp(y, j_k) \to \bar{\mu}$ for some subsequence $(j_k)$, then also $\mu_{n(j_k)} \to \bar{\mu}$. Conversely, if $\mu_{n_k} \to \bar{\mu}$ for some $\bar{\mu}$, then $\Emp(y, L_{n_k - 1}) \to \bar{\mu}$.

It remains to address the ``moreover'' part of the lemma statement:
\begin{enumerate}
  \item Regardless of the sequence $(\mu_n)$, by (\ref{eqn:yn_close_y}) we have $\rho(y, y_0) \le \eps / 2$.
  \item Note that construction of $y_1 \dots y_n$ depends only on $\mu_1 \dots \mu_n$. Consider another sequence $(\mu_n')$, such that $\mu_j = \mu_j'$ for $j \le n$. Let $y'$ be the point obtained from this sequence. Then, since $y_n = y_n'$, we deduce from \eqref{eqn:yn_close_y} that
  \[
    \rho(y, y') \le \rho(y, y_n) + \rho(y_n', y') \le \eps /2^{n+1} + \eps /2^{n+1} = \eps / 2^n.
  \qedhere\]
\end{enumerate}
\end{proof}

We make the following observation:
\begin{cor} \label{cor:dense-gv}
If $(X,T)$ is  a Polish dynamical system with the strong approximate product structure
and $V\subseteq \MT(X)$ is a nonempty closed connected set, then the set $G(V)$ is dense in $X$.
For any $T$-invariant measure $\mu$, the set $G_\mu$ is dense.
\end{cor}

\begin{proof}
For the first part, note that there exists a sequence $(\mu_n)_{n \ge 1}$ such that $D(\mu_n, \mu_{n+1}) \to 0$ and $\{ \textrm{limit points of } (\mu_n)_{n \ge 1} \} = V$. Thus an application of Lemma \ref{lem:main} shows that $G(V)$ is dense. Similarly, application of Lemma \ref{lem:main} to the constant sequence $\mu_n = \mu$ shows that $G_\mu$ is dense.
\end{proof}

\begin{thm} \label{thm:main}
If a Polish dynamical system $(X,T)$ has the strong approximate product structure,
then

\begin{enumerate}
\item \label{main1}
If $|\MT(X)|\geq 2$, then
for any $\mu \in \MT(X)$
the set $G_\mu$ of $\mu$ generic points is $\bP^0_3$-complete.
\end{enumerate}
Moreover, we have the following:
\begin{enumerate} \setcounter{enumi}{1}
\item \label{main2}
For every nonempty closed connected set $V\subseteq \MT(X)$ with nonempty
complement $U=\MT(X)\setminus V$, the sets $G_V$ and $G(V)$ are $\bP_3^0$-hard
(and so the  set $\prescript{U}{}G$ is $\bS_3^0$-hard).
Also, the set $G^V$ is $\bP^0_2$-complete.
The sets $G_V$, $G^V$, $G(V)$, and $\prescript{U}{}G$ are all dense in $X$.
\end{enumerate}
\end{thm}

\begin{proof}
We give the proof of (ii),
and remark at the end how the proof shows (\ref{main1}). Let $V\subseteq \MT(X)$
be as in (ii) and let $U=\MT(X) \setminus V$.
Let $\cC_3\subseteq  \omega ^ \omega$ be the set $\cC_3=\{ \beta\in  \omega ^ \omega \colon \beta(n) \to \infty\}$.
We construct a continuous reduction map
$\pi\colon \omega^\omega\to X$ such that if $\beta \in \mathcal{C}_3$ then $\pi(\beta)\in G(V)\subseteq G_V$,
and if $\beta\notin \mathcal{C}_3$ then $\pi(\beta)\in \prescript{U}{}G=\MT(X)\setminus G_V$.
This will show that $G_V$ and $G(V)$ are $\bP^0_3$-complete.

Take $\nu \in U$ and let $\mu'\in V$. Consider the set
$E=\{0\le t\le 1: t\mu'+(1-t)\nu\in V\}$. It is easy to see that $E\neq \emptyset$ is closed, $1\in E$,
and $0\notin E$. Let $t_0=\min E$ and $\bar\mu=t_0\mu'+(1-t_0)\nu\in V$. 
Hence $\alpha\bar\mu+(1-\alpha)\nu\in U$ for $\alpha<1$.
Since $V$ is closed and connected and $\MT(X)$ is separable there is a
sequence $(\mu_n)_{n\ge 1}$ of $T$-invariant measures satisfying
\begin{gather}
  D(\mu_n,\mu_{n+1})\to  0\quad\text{ as $n\to\infty$}, \label{cond:D-to-0}\\
  \bigcap_{N=1}^\infty\overline{\{\mu_n:n\ge N\}}=V .\label{cond:density-in-V}
\end{gather}
That is, the set of limit points of sequence $(\mu_n)_{n \ge 1}$ is equal to $V$.

Select a sequence $(n_k)_{k\ge 1}$ such that $\mu_{n_k}\to\bar\mu$ as $k\to\infty$, and additionally $n_{k+1} - n_k \to \infty$. Given a $\beta \in \omega ^ \omega$, we define a function
\begin{equation} \label{eqn:psibeta}
\psi^\beta (j) = \frac{j - n_k}{n_{k+1} - n_k} \frac{1}{\beta(k) + 1} + \frac{n_{k+1} - j}{n_{k+1} - n_k} \frac{1}{\beta(k+1) + 1}
\quad \textrm{ for } n_k \le j < n_{k+1}
\end{equation}
and a sequence of measures
\[
\mu^\beta_j = \psi^\beta (j) \nu + (1 - \psi^\beta (j) ) \mu_j.
\]
Note that $\psi^\beta(n_k) = 1/\beta(k)$ for all $k$, and $\psi^\beta$ is ``linear'' between $n_k$ and $n_{k+1}$. Thus we have
\begin{equation*}
| \psi^\beta(j) - \psi^\beta(j+1) | \le \frac{1}{n_{k+1} - n_k} \to 0 \quad \textrm{ as } j \to \infty.
\end{equation*}
It follows that $D(\mu^\beta_j, \mu^\beta_{j+1}) \to 0$ as $j \to \infty$. Thus we can apply Lemma \ref{lem:main} to the sequence $(\mu^\beta_j)$, and obtain a point $\pi(\beta)$ such that $V(\pi(\beta)) = \{ \textrm{ limit points of } \mu^\beta_j \}$. Let us consider cases:

\begin{itemize}
  \item If $\beta(n) \in \mathcal{C}_3$, then $\psi^\beta(j) \to 0$ as $j \to \infty$. It follows that $D(\mu^\beta_j, \mu_j) \to 0$, and consequently $V(\pi(\beta)) = \{ \textrm{ limit points of } \mu^\beta_j \} = \{ \textrm{ limit points of } \mu_j \} = V$.
  \item If $\beta(n) \not \in \mathcal{C}_3$, then there exists a subsequence $(m_k)$ such that $\beta(m_k) = c$. Then
  \[
    \mu^\beta_{n_{m_k}} = \frac{1}{c} \nu + \frac{c-1}{c} \mu_{n_{m_k}} \to \frac{1}{c} \nu + \frac{c-1}{c} \bar{\mu}
    \not \in V
  \]
  and it follows that $V(\pi(\beta))  \not \subseteq V$.
\end{itemize}
Finally, notice that $\mu^\beta_1, \dots \mu^\beta_{n_k}$ depend only on $\beta(1), \dots \beta(n_k)$. Thus the map $\pi$ is continuous by  Lemma \ref{lem:main}\eqref{it:iii}. This shows that $G(V)$ and $G_V$ are $\bP^0_3$-hard, and $^U G$ is $\bS_3^0$-hard. Note that Lemma \ref{lem:main}\eqref{it:ii} shows that all those sets are dense.

Suppose now that $|\MT(X)|\geq 2$, and let $\mu \in \MT(X)$. We
reapply the above proof for the case $V=\{ \mu\}$. Since $|\MT(X)|\geq 2$,
$U=\MT(X)\setminus V \neq \emptyset$. We can take the measures
$(\mu_n)$ of the previous proof to be all equal to $\mu$.
We also let $\nu \in U=\MT(X)\setminus V$. For $\beta \in \mathcal{C}_3$,
the previous argument shows in this case that
not only is $\pi(\beta) \in G(\{ \mu\})$, but in fact
$\pi(\beta) \in G_\mu$. This follows from \eqref{eqn:lemmain} in Lemma \ref{lem:main}\eqref{it:i}. For $\beta \notin \mathcal{C}_3$, the previous argument
gives that $\pi(\beta) \notin G(\{\mu\})$, and so $\pi(\beta) \notin G_\mu$.
Thus, $\beta \in \mathcal{C}_3$ iff $\pi(\beta)\in G_\mu$, and so
$G_\mu$ is $\bP^0_3$-hard (and thus $\bP^0_3$-complete by Fact~\ref{upper:Gmu}).

To prove that $G^V$ is $\bP^0_2$-complete, first recall that it is a $\bP^0_2$ set (Fact \ref{upper:GW}). Again by Lemma \ref{lem:main}\eqref{it:ii}, we see that $G^V$ is dense. Given $\nu \notin V$, we know that $G_\nu$ is dense. Note that $G_\nu \subseteq X \setminus G^V$, thus both $G^V$ and its complement are dense. It follows that $G^V$ is $\bP^0_2$-complete (indeed, suppose that $G^V$ is $\bS^0_2$-complete; then both $G^V$ and its complement would be dense $\bP^0_2$ sets, which is impossible by the Baire Category Theorem).
This completes the proof of Theorem~\ref{thm:main}.
\end{proof}

Combined with the upper bounds of Facts~\ref{upper:UG} and \ref{upper:GV},
Theorem~\ref{thm:main} has the following immediate corollary.

\begin{cor}
If $(X,T)$ is a compact dynamical system with the strong approximate product structure
and $V\subseteq \MT(X)$ is a nonempty closed connected set whose complement
$U=\MT(X)\setminus V$ is also nonempty, then
\begin{enumerate}
\item the sets $G_V$ and $G(V)$ are $\bP_3^0$-complete,
\item the set $\prescript{U}{}G$  is $\bS_3^0$-complete.
\end{enumerate}
\end{cor}

Observe that for any ergodic probability measure $\mu$, it follows from the
Birkhoff ergodic theorem that
for $\mu$-almost all $x$ we have $x\in B_\phi(\int\phi\,d\mu)$.

%
%
%
%
%
%
%

\begin{thm} \label{thm:phi_nocompact}
Let $(X,T)$ be a Polish dynamical  system with the strong approximate product structure.
If $\phi\in\mathcal{C}_b(X)$ and $\alpha\in\R$ are such that
\[
V_\phi(\alpha)=\{\mu\in\MT(X)\colon \int\phi\,d\mu=\alpha\}\neq\emptyset
\]
and $U=\MT(X)\setminus V_\phi(\alpha)\neq\emptyset$, then
$B_\phi(\alpha)$ is a $\bP^0_3$-complete set and $I_\phi(X,T)$ is a $\bS^0_3$-complete set.
\end{thm}

\begin{proof}
Pick any measure $\nu \in U$. As in the proof of Theorem \ref{thm:main}, we can pick a measure $\bar{\mu} \in V_\phi(\alpha)$ such that $t \nu + (1-t) \bar{\mu} \in U$ for all $0 < t \le 1$. We select an arbitrary sequence $n_k \nearrow \infty$, with $n_{k+1} - n_k \to \infty$. Given a sequence $\beta \in  \omega ^ \omega$, we define function $\psi^\beta \colon \omega \to [0, 1]$ by setting $\psi^\beta(n_{2k}) = 1/(\beta(k) + 1)$, $\psi^\beta(n_{2k+1}) = 0$, and interpolating linearly on other arguments, that is,
\begin{equation}
\psi^\beta (j) = \frac{j - n_k}{n_{k+1} - n_k} \psi^\beta(n_k) + \frac{n_{k+1} - j}{n_{k+1} - n_k} \psi^\beta(n_{k+1}).
\quad \textrm{ for } n_k < j < n_{k+1}.
\end{equation}

It is straightforward to check that
\begin{enumerate}
  \item $|\psi^\beta(j) - \psi^\beta(j+1)| \to 0$ as $j \to \infty$,
  \item if $\beta \to \infty$, then $\psi^\beta(j) \to 0$,
  \item if not $\beta \to \infty$, then both $0$ and $1 / (\liminf \beta(j) + 1)$ are limit points of sequence $\psi^\beta(j)$.
\end{enumerate}
We define measures
\[
\mu^\beta_j = \psi^\beta(j) \nu + (1 - \psi^\beta(j)) \bar{\mu}.
\]
From the properties of $\psi^\beta$ listed above it follows that
\begin{enumerate}
  \item $D(\mu^\beta_j, \mu^\beta_{j+1}) \to 0$ as $j \to \infty$, thus we can invoke Lemma \ref{lem:main} to obtain a point $y^\beta$,
  \item if $\beta \to \infty$, then $\mu_j^\beta \to \bar{\mu}$ and consequently $y^\beta \in G_{\bar{\mu}} \subseteq B_\phi(\alpha)$,
  \item if not $\beta \to \infty$, then $y^\beta \in I_\phi(X,T)$.
\end{enumerate}
Note that the measures $\mu^\beta_1 \dots \mu^\beta_{n_{2k}}$ depend only on the values $\beta(1) \dots \beta(k)$. This, together with part (ii) of Lemma \ref{lem:main} implies that the map $\pi \colon  \omega ^ \omega \ni \beta \mapsto y^\beta \in X$ is continuous. By (ii) and (iii), $\pi$ is a continuous reduction of $\mathcal{C}_3$ to $B_\phi(\alpha)$, and also a reduction of $ \omega ^ \omega \setminus \mathcal{C}_3$ to $I_\phi(X,T)$. Together with Fact \ref{upper:irregular} this finishes the proof.
\end{proof}
From the proof above, we also obtain the following corollary:

\begin{cor}
If $(X,T)$ is  a Polish dynamical system with the strong approximate product structure
with respect to a compatible metric $\rho$ and there are at least two $T$-invariant measures, then
\begin{enumerate}
\item the quasi-regular set $Q(X,T)$ is $\bP_3^0$-hard, thus
$Q(X,T)$ is $\bP^0_3$-complete.
\item the irregular set $I(X,T)$ is $\bS_3^0$-hard, and thus
$I(X,T)$ is $\bS_3^0$-complete.
\end{enumerate}
\end{cor}

\begin{proof}
Pick any two $T$-invariant measures $\bar{\mu} \neq \nu$ and a function $\phi \in \mathcal{C}_b(X)$ such that $\alpha = \int_X \phi \, d \mu \neq \int_X \phi \, d \nu$. In the proof of Theorem \ref{thm:phi_nocompact}, we defined a continuous map $\pi \colon  \omega ^ \omega \to X$ such that $\pi(\mathcal{C}_3) \subseteq G_{\bar{\mu}}$ and $\pi( \omega ^ \omega \setminus \mathcal{C}_3) \subseteq I_\phi(X,T)$. Since $G_{\bar{\mu}} \subseteq Q(X,T)$ and $ I_\phi(X,T) \subseteq I(X,T)$, the map $\pi$ reduces $\mathcal{C}_3$ to $Q(X,T)$, and reduces $ \omega ^ \omega \setminus \mathcal{C}_3$ to $I(X,T)$. In Fact \ref{upper:irregular} we have shown the matching upper bounds.
\end{proof}

\section{Compactness and Higher level completeness} \label{sec:hlc}

In Theorem~\ref{thm:main} it was shown for a Polish dynamical system $(X,T)$ satisfying
a weak form of specification (the strong approximate product structure),
and having at least two invariant measures, that for
every closed connected $V \subseteq \cM_T$ which is non-trivial (i.e., $V\neq \emptyset$, $V\neq \cM_T$),
we have that $G(V)$ is $\bp^0_3$-hard. When $X$ is compact, Fact~\ref{upper:GV} shows that
$G(V)$ is a $\bp^0_3$ set (this holds for any closed $V\subseteq \MT(X)$), and so
$G(V)$ is a $\bp^0_3$-complete set in this case. Similar results hold for the sets
$G_V=\{ x \in X\colon V(x) \subseteq V\}$ (where $V\subseteq \MT(X)$ is closed) and $\prescript{U}{}G=\{x\in X: V(x)\cap U\neq\emptyset\}$ (where $U\subseteq \MT(X)$ is open).
When $X$ is not compact, the natural computation shows that
$G(V)$ and $G_V$ are $\bp^1_1$ sets and $\prescript{U}{}{G}$ is a $\bs^1_1$ set (see Facts \ref{fact:UG-analytic} and~\ref{GV_co-analytic}). Recall from Fact~\ref{upper:GW} that in this case we still have that
for any closed $V\subseteq \cM_T$ that $G^V=\{ x \in X \colon V\subseteq V(x)\}$ is $\bp^0_2$.
A natural question is whether the compactness of $X$ is necessary to reduce the complexity of $G_V$, $G(V)$, and $\prescript{U}{}{G}$
to a Borel set. The next result shows that this is the case, as in general
the sets $G(V)$ and $G_V$ may be a $\bp^1_1$-complete sets, while $\prescript{U}{}{G}$ is a $\bp^1_1$-complete set. Thus, we see a strong dichotomy in the complexity of
these sets between the compact and non-compact cases.

\begin{thm} \label{thm:ncc}
There is a Polish dynamical system $(X,T)$, so $T \colon X\to X$ is continuous, and
a closed, connected, nonempty $V \subseteq \MT(X)$ such that $G(V)$ is $\bp^1_1$-complete. In fact,
we may take $V=\{ \mu\}$ for some $\mu \in \MT(X)$, and $(X,T)$ to have the specification property.
Furthermore, setting $U=\{\nu\in\MT(X):\nu\neq\mu\}$ we obtain
that $G_V$ is $\bp^1_1$-complete and $\prescript{U}{}G$ is $\bs^1_1$-complete.
\end{thm}

\begin{rem} \label{nts}
For $V=\{\mu\}$, where $\mu \in \MT(X)$, the sets $G_\mu$ and $G(V)$ may not be equal. For
$x$ to be in $G_\mu$, that is, for $x$ to be a $\mu$ generic point, we must have that
for all increasing sequences of positive integers $(n_i)_{i\ge 1}$ that the empirical measures $\cE(x,n_i)$ converge to
$\mu$. For $x$ to be in $G(V)$, we need only have that for each such sequence $n_i$
that if the empirical measures $\Emp(x,n_i)$ converge to a measure, then that measure must be $\mu$.
This distinction is important, since the set $G_\mu$ is always $\bp^0_3$ by Fact~\ref{upper:Gmu}.
\end{rem}

\begin{proof}[Proof of Theorem~\ref{thm:ncc}]
Let $X=\omega^\omega$ with the (continuous) shift map $T$. For $n\in\omega$
write $\delta_{\bar n}$ for the point-mass measure on the constant $n$
sequence denoted $\bar{n}$. Let $V=\{ \delta_{\bar 0}\}$.

To show that $G(V)$ is $\bp^1_1$-complete, we construct a continuous map
$\pi\colon\tr\to X$ reducing $\wf$ to $G(V)$, where $\wf$
is the set of all wellfounded trees on $\omega$. Recall that a tree $\mathcal{T}$ on $\omega$ is a
set $\mathcal{T}\subseteq \omega^{<\omega}$ of finite length sequences from $\omega$ which is
closed under the operation of taking initial segments. A tree is called \emph{wellfounded} if it doesn't contain an infinite branch, and \emph{illfounded} if it does. Let $\tr$ denote the set of all trees on $\omega$.
We fix an enumeration $t^{(1)},t^{(2)},\ldots$ of $\omega^{<\omega}$ such that
$|t^{(i)}|\le |t^{(j)}|$ for $i\le j$. Using this enumeration we identify
$\omega^{<\omega}$ with $\omega$ and $\tr$ with a closed subset of the
compact Polish space $2^\omega$.
Thus, $\tr$ is a Polish space and $\wf\subseteq \tr$ is a $\bp^1_1$-complete set
in $\tr$ (see \cite[ch. 32.B]{Kechris}). Note that two trees in $\tr$ are close
if and only if there is $n$ such that for each $1\le i \le n$, the
finite word $t^{(i)}$ either belongs to both trees or belongs to neither of them.


First, we will associate to each finite word $s\in\omega^{<\omega}$
a measure $\mu_s$, which will be a
finite combination of Dirac point-mass measures.
Let us fix some auxiliary notation. By $p(n)$ we will denote the $n$-th prime
(that is, $p(0)=2$, $p(1)=3$ etc.). Fix a bijection
$\mathbf{n}\colon\omega\times\omega\to\omega$. Let $Q$ be the set of
positive integers divisible by at least two different primes.
Let $q(n)$ denote the $n$-th smallest element of $Q$. That is, $q(0)=6$, $q(1)=10$, $q(2)=12$ etc.\
Given $(n,i)\in\omega\times\omega$ we set $\nu_{n,i}$ to be the Dirac point-mass
measure concentrated on the constant sequence $\bar p=ppp\ldots\in\omega^\omega$
where $p=p(\mathbf{n}(n,i))$.
For $s=(s(0)\ldots s(\ell-1))\in\omega^{<\omega}$ (so $\ell=|s|$ is the length of $s$)
we define a measure $\mu_s$ as
\[
\mu_s=\frac{1}{2^\ell}\delta_{\bar{1}}+\sum_{i<\ell} \frac{1}{2^{i+1}}\nu_{s(i),i}.
\]

Suppose $s_0, s_1, s_2,\dots$ is a sequence from $\omega^{<\omega}$
with $\lim_i |s_i|=\infty$ and such that $s_i$ converges to $x=(x(0),x(1),\dots) \in \ww$,
that is, $\forall m\ \exists n\ \forall i\geq n\ (s_i\res m=x\res m)$.
Then
the corresponding sequence of measures $\mu_{s_i}$ is easily seen to converge to the
measure $\mu_{x}$ as $i\to\infty$,  where $\mu_x$ is the measure given by
\[
\mu_{x}=\sum_{i=0}^\infty\frac{1}{2^{i+1}}\nu_{x(i),i}=
\sum_{i=0}^\infty\frac{1}{2^{i+1}}\nu_{p(\mathbf{n}(x(i),i))}.
\]

On the other hand assume that  $s_0, s_1, s_2,\ldots$ is a sequence of distinct
elements of $\omega^{<\omega}$ and the corresponding sequence of measures
$\mu_{s_i}$ is Cauchy. We claim that $\lim_i |s_i|=\infty$ and there is an $x \in \ww$
such that $s_i$ converges to $x$. Let $k$ be the least integer, if one exists,
such that either there are infinitely many $i$ with $s_i(k)$ not defined or
$\{ s_i(k) \}$ is not eventually constant. Assuming $k$ is defined, the first case cannot
happen as then $s_i\res k$ is defined and constant for all large enough $i$ but
for some $i<j$ we would have $s_i=s_j$. If the second case
in the definition of $k$ occurs, then for large enough $i$ we have
that $s_i\res k$ is constant and either  $s_i(k)$ takes infinitely many values or
there are two distinct values of $s_i(k)$ both of which occur for infinitely
many values of $i$.
However, we would then have
$D(\mu_{s_i}, \mu_{s_j})\geq 1/2^{k+1}$ for arbitrarily large $i<j$,  a contradiction.
So, $k$ is not defined, that is,
for all $k$ we have that for all large enough $i$ that $s_i(k)$ is defined
and constant. This shows that the $s_i$ converge to some $x \in \ww$.

Finally, for each $s=(s(0)\ldots s(\ell-1))\in\omega^{<\omega}$ and $0\le i<\ell$
we set $n^s_i=p(\mathbf{n}(s(i),i))$ and define a periodic point
\[
y^s=
\left(
\underbrace{n^s_0\ldots n^s_0}_{2^{2\ell}\text{ times}}\
\underbrace{n^s_1\ldots n^s_1}_{2^{2\ell-1}\text{ times}}\  \ldots\
\underbrace{n^s_{\ell-1}\ldots n^s_{\ell-1}}_{2^{\ell+2}\text{ times}}\
\underbrace{1\ldots 1\vphantom{n^s_1}}_{2^{\ell+1}\text{ times}}
\right)^\infty\in\omega^\omega.
\]
Note that the primary period of $y^s$ equals $p_s=2^{2\ell+1} - 2^{\ell + 1}$ and
$D(\Emp(y^s, p_s),\mu_s)$ goes to $0$ as $\ell$ goes to infinity.
We are ready to define $\pi(\mathcal{T})$ for a tree $\mathcal{T}$.
Let $c_n$ be the primary period of $y^s$ for $s=t^{(n)}$.
We fix  
sequences $( a_n)$ and $(b_n)$ satisfying
$a_1=2$, $a_{2n}=c_n$, $a_{2n+1}=1$, $b_{n+1} > b_n$ and
\begin{equation} \label{fqa}
a_nb_n> \max\{ 2^n a_{n+1}, 2^n (a_0 b_0+\cdots+a_{n-1}b_{n-1})\}
\end{equation}
for all $n$.


We set $\ell_0=0$. As stage $n\ge 1$ of the construction of $x=\pi(\beta)$
we will define $x_{[\ell_{n-1},\ell_n)}$ where $\ell_n=\ell_{n-1}+ a_n b_n$.
For $n$ odd we set $x_{[\ell_{n-1},\ell_n)}={\bar{0}}_{[\ell_{n-1},\ell_n)}=0\ldots 0$
(here $0$ appears $a_nb_n$ times). For $n$ even,
we have two cases: if $s=t^{(n/2)}\in \mathcal{T}$, then we put $x_{[\ell_{n-1},\ell_n)}=(y^s)^{b_n}$.
Otherwise (that is, if
$s=s^{(n/2)}\notin \mathcal{T}$), we put $x_{[\ell_{n-1},\ell_n)}=(q(n/2))^{a_nb_n}$.
Note that $\pi(\mathcal{T})=x$ defined in this way depends continuously on $\mathcal{T}$
because $x_{[0,\ell_{2n})}$ is determined by
$\{t^{(1)},\ldots,t^{(n)}\}\cap\mathcal{T}$. Furthermore, it follows from (\ref{fqa}) that the
measures $\Emp(x, \ell_{2n-1})$ converge to $\delta_{\bar 0}$ as $n\to\infty$.
Thus $\delta_{\bar 0} \in V(x)$.
We show that $\delta_{\bar 0}$ is the only measure in $V(x)$ if and only if
$\mathcal{T} \in \tr$ is a wellfounded tree.
Let $\nu \in V(x)$ and $i_k$ be a strictly increasing sequence with
$\lim_{k\to\infty} \cE(x, i_k)=\nu$. Let $(n_k)$
be a sequence of integers such that $\ell_{n_k} \leq i_k < \ell_{n_{k}+1}$ for every $k$.
Without loss of generality we
assume that the $(n_k)$ is strictly increasing and $n_k$ is even for every $k$.
The proof in the other case is similar.

Let 
$e_k=(i_k-\ell_{n_k})/i_k$.
By passing to a subsequence we may assume that $e_k\to e\in [0,1]$ as $k\to\infty$. From
(\ref{fqa}) we have that the sequence $\Emp(x, i_k)$ and
the sequence with $k$-th term given by
\[\nu_k=(1-e) \Emp(x_{[\ell_{n_{k}-1},\ell_{n_k})})+ e\Emp(x_{[\ell_{n_{k}},\ell_{n_k+1})})\]
has the same weak${}^*$ limit in $\cM_T(X)$.
If $e=1$, then $\nu=\delta_{\bar{0}}$, so assume $e<1$.
Take the sequence $(m_k)$ so that $n_k=2m_k$ for every $k$.

Note that $x_{[\ell_{n_{k}},\ell_{n_k+1})}$ is a sequence containing only zeros. Also,
\begin{equation} \label{eqn:konrad19}
x_{[\ell_{n_{k}-1},\ell_{n_k})}=
  \begin{cases}
    q_k^{a_{n_k}b_{n_k}}, \quad & \textrm{ where } q_k=q(m_k),\textrm{ if } t^{(m_k)}\notin\mathcal{T}, \\
    (y^s_{[0, a_{n_k})})^{b_{n_k}}, \quad & \textrm{ where } s = t(m_k), \textrm{ if }  t^{(m_k)} \in \mathcal{T}.
  \end{cases}
\end{equation}
It follows that the sequence $\nu_k$ converges to the same weak* limit as the sequence of measures

\begin{equation} \label{eqn:konrad20}
w_k =
  \begin{cases}
    (1 - e) \delta_{\overline{q_k}} + e \delta_{\bar{0}}, \quad & \textrm{ if } t^{(m_k)}\notin\mathcal{T}, \\
    (1 - e) \mu_{t^{(m_k)}} + e \delta_{\bar{0}}, & \textrm{ if }  t^{(m_k)} \in \mathcal{T}.
  \end{cases}
\end{equation}
Suppose $t^{(m_k)}\notin\mathcal{T}$ for some subsequence $(k(j))_{j \ge 1}$. Then the sequence of measures $\delta_{\overline{q_{k(j)}}}$ would also converge - contradiction.
Thus for sufficiently big $k$, the second case of (\ref{eqn:konrad20}) holds. It follows that the sequence
$(\mu_{t^{(m_k)}})$ converges.
%
This, by our previous
observations, gives that the $t^{(m_k)}$ converge to some $z \in \ww$.
Since each $t^{(m_k)} \in \mathcal{T}$, this show that $z$ is a branch through $\mathcal{T}$
and so $\mathcal{T}$ is illfounded. We have shown that if $V(\pi(\mathcal{T})) \neq \{ \delta_{\bar 0}\}$,
then $\mathcal{T}$ is illfounded. Conversely, if $\mathcal{T}$ is illfounded,
say $z$ is a branch through $\mathcal{T}$, then along the sequence $(n_{2k})$ where
$t^{(k)}=z\res k$ we have that $\Emp(\pi(\mathcal{T}), n_{2k})$ converges to the measure $\mu_z$
defined above, and $\mu_z \neq \delta_{\bar 0}$, and so we have $V(\pi(\mathcal{T}))
\neq \{ \delta_{\bar 0}\}$. So $V(\pi(\mathcal{T}))=\{\delta_{\bar{0}}\}$
if and only if $\mathcal{T}\in\wf$. This finishes the proof. It is easy to see that the above proof shows in fact that $G_{\{\delta_{\bar{0}}\}}$ is $\bp^1_1$-complete, while $\prescript {U}{}{G}$, where $U=\MT(X)\setminus\{\delta_{\bar{0}}\}$ is $\bs^1_1$ complete.
\end{proof}

\section{Complexity of $G_{\mu}$ in a minimal subshift} \label{sec:minimalsubshift}
\newcommand{\one}{\mymathbb{1}}
\newcommand{\zero}{\mymathbb{0}}
\newcommand{\x}{\mathtt{x}}
\newcommand{\suf}{\sqsubseteq}

Recall that a dynamical system $(X, T)$ is called \emph{minimal}, if $\overline{ \{T^n x \colon n \ge 0 \} } = X$ for
every $x \in X$. Minimality is a stronger condition than transitivity. We give an example that shows that $G_\mu$ may be $\bP_3^0$-complete in a minimal system.

\begin{thm} \label{thm:minimal_pi03}
  There exists a minimal shift $(X, T)$ and an ergodic measure $\mu$ such that $G_\mu$ is $\bP_3^0$-complete.
\end{thm}
We will examine a construction given by Oxtoby, originally used as an example of a minimal, not uniquely ergodic dynamical system; see \cite[Example 10.3]{Downarowicz} or \cite[§10]{Oxtoby}.
We will construct a subshift of $\{0, 1\}^\omega$.  The construction uses a sequence $(s_j)_{j \geq 1}$ of natural numbers such that
\begin{equation} \label{eqn:s_j_condition}
  \sum_{j=1}^{\infty} \frac{1}{s_j} < 1.
\end{equation}
We will now define certain strings over the alphabet $\{0, 1, \x\}$.
Here $\x$ is a placeholder to be filled at a later stage of the construction.
We begin with
\begin{subequations} \label{eqn:recurrence_def}
	\begin{align}
		W_0 &= \x,
		& \zero_0 &= 0,
		& \one_0 &= 1, \\
		\nonumber
		\textrm{for odd n:}\\
		W_n &= \zero_{n-1}W_{n-1}^{s_n - 1},
		& \zero_n &= \zero_{n-1}^{s_n},
		& \one_n &= \zero_{n-1}\one_{n-1}^{s_n - 1}, \\
		\nonumber
		\textrm{for even n:}\\
		W_n &= \one_{n-1}W_{n-1}^{s_n - 1},
		& \zero_n &= \one_{n-1}\zero_{n-1}^{s_n - 1},
		& \one_n &= \one_{n-1}^{s_n}.
	\end{align}
\end{subequations}
To illustrate, the sequence $(W_n)$ begins with
\begin{align*}
	W_0 &= \mathtt{x}, \\
	W_1 &= \mathtt{0 x x}, \\
	W_2 &= \mathtt{0 1 1 0 x x 0 x x 0 x x}, \\
	W_3 &= \mathtt{0 1 1 0 0 0 0 0 0 0 0 0 0 1 1 0 x x 0 x x 0 x x 0 1 1 0 x x 0 x x 0 x x}.
\end{align*}
For all $n$, the word $\zero_n$ can be obtained by replacing all $\x$ symbols with zeros in $W_n$.
Similarly, the word $\one_n$ can be obtained by replacing all $\x$ symbols with ones in $W_n$.
The lengths of $W_n, \zero_n$ and $\one_n$ are equal, we will denote this length by $l_n$.
Observe that in the sequence $\zero_0, \one_1, \zero_2, \one_3, \dots$ each word is a prefix of the next one,
and thus we can define $y \in \{0, 1\}^\omega$ to be the limit of this sequence.
Finally, we define $X$ to be the closure of the orbit of $y$.
The sequence $y$ is Toeplitz, and $X$ is a minimal subshift; see \cite{Downarowicz} for details.
Note that the language of $X$ can be described as follows:
$$
	L = L(X) = \{ u \in \{0, 1\}^{< \omega} \colon u \textrm{ is a subword of } \zero_n
	\textrm{ or } \one_n \textrm{ for some } n \ge 0 \}.
$$
For a word $u$, define $c(u) = | \{ 0 \le i < |u| \colon u_i = 1\}|$ and $m(u) = c(u) / |u|.$
From (\ref{eqn:recurrence_def}), it follows that $m(\zero_n)$ is a weakly increasing sequence, and therefore $ m(\zero_n) \nearrow a$ for some real number $a \in [0, 1]$. Similarly, $m(\one_n) \searrow b$ for some $b \in [0, 1]$. Again from (\ref{eqn:recurrence_def}), we get $m(\one_n) - m(\zero_n) = \prod_{j=1}^n (1 - 1/s_j)$.
The condition (\ref{eqn:s_j_condition}) implies that $\prod_{j=1}^\infty (1 - 1/s_j) > 0$, and consequently $a < b$.

\begin{fact} \label{fact:twoergmeasures}
	The set $\MTe(X)$ consists of two measures $\mu_0, \mu_1$. Moreover $\mu_0(C) = a$ and $\mu_1(C) = b$,
	where $C = \{ x \in X : x_0 = 1 \}$.
\end{fact}
\begin{proof}
$X$ has exactly two ergodic measures \cite[Chapter 14]{Downarowicz}.
Note that for arbitrary continuous $f \colon X \to \R$, we have
\begin{equation} \label{eqn:measureofcyl1}
	\min_{\mu \in \MTe(X)} \int_X f d \mu =
	\min_{\mu \in \MT(X)} \int_X f d \mu =
	\lim_{k \to \infty} \min_{x \in X} A_k f(x).
\end{equation}
We apply (\ref{eqn:measureofcyl1}) to the characteristic function $\chi_C$.
We claim that the right side of (\ref{eqn:measureofcyl1}) is then equal to $a$. Indeed, note that in this particular case, the right side simplifies to
\begin{equation} \label{eqn:minoverwords}
	\lim_{k \to \infty} \min_{u \in L_k(X)} m(u),
\end{equation}
where $L_k(X) = \{ u \in L(X) \colon |u| = k \}$.
Pick any $n \ge 0$, then any word in $L(X)$ can be divided into blocks,
each equal to $\zero_n$ or $\one_n$, where the first and last block may be incomplete.
This shows that (\ref{eqn:minoverwords}) is greater or equal to $m(\zero_n)$.
Since this is true for any $n$, we get that (\ref{eqn:minoverwords}) is greater or equal to $a$.
On the other hand, by plugging $u = \zero_n$, we find that (\ref{eqn:minoverwords}) is less or equal than $a$.
Thus there exists a measure $\mu_0 \in \MTe(X)$ such that $\mu_0(C) = a$.
Similarly, applying (\ref{eqn:measureofcyl1}) to $- \chi_C$ shows us that
there exists a measure $\mu_1 \in \MTe(X)$ such that $\mu_1(C) = b$.
\end{proof}

\begin{fact} \label{fact:zeroseq}
	For any $k$, suppose we have $j_0, j_1, \dots j_k$ such that $0 \leq j_i \leq s_{i+1} - 2$.
	Then the word $\zero_0^{j_0} \zero_1^{j_1} \dots \zero_k^{j_k}$ is a suffix of $\zero_{k+1}$,
	in particular it belongs to the language $L$.
\end{fact}
\begin{proof}
	Let us write $u \suf v$ whenever the word $u$ is a suffix of $v$.
	We will proceed by induction.
	For $k=0$, the statement is true.
	Now, suppose the statement holds for $k-1$, then we have
	$\zero_0^{j_0} \zero_1^{j_1} \dots \zero_{k-1}^{j_{k-1}} \suf \zero_k$.
	From this, we obtain
	$
		0_0^{j_0} 0_1^{j_1} \dots 0_{k-1}^{j_{k-1}} 0_{k}^{j_{k}} \suf
		0_k^{j_k + 1} \suf 0_k^{s_{k+1} - 1} \suf 0_{k+1}.
	$
	Since the relation $\suf$ is transitive, we are done.
\end{proof}
From the fact above, it follows that the infinite sequence $0_0^{j_0} 0_1^{j_1} \dots 0_k^{j_k} \dots$ is an element of $X$, if $1 \leq j_i \leq s_{i+1} - 2$ for all $i$. Recall that the set
$ \mathcal{C}_3 := \{ \beta \in \omega ^ \omega : \liminf \beta(n) = +\infty \} $
is a $\bP_3^0$-complete subset of $\in \omega ^ \omega$. We define $ f : \omega ^ \omega \rightarrow X $ by the formula
$$ f(\beta) = \zero_1^{j_0} \zero_3^{j_1} \dots \zero_{2i+1}^{j_i} \dots, \quad \textrm{ where } j_i = \min\{ \beta(i) + 1, s_{2i+2} - 2 \}.$$
We claim that $\beta \in \mathcal{C}_3 \Leftrightarrow f(\beta) \in G_{\mu_0}$.
Fact \ref{fact:twoergmeasures} implies that $G_{\mu_0} = B_{\chi_C}(a)$.
This will simplify our analysis considerably.
\begin{itemize}
	\item[$\beta \not \in \mathcal{C}_3$:]
		Let $k = \liminf \beta(i) + 1$. Consider any $i$ such that $j_i = k$.
		From (\ref{eqn:s_j_condition}) note that $\one_{2i + 1}$ is a prefix of $\zero_{2i+3}$. Thus
		$$
			w =  \underline{\zero_1^{j_0} \zero_3^{j_1} \dots \zero_{2i-1}^{j_{i-1}}} \zero_{2i+1}^{k} \one_{2i + 1}
		$$
		is a prefix of $f(\beta)$. The underlined part is a suffix of $\zero_{2i}$,
		and hence $|w| \le |\zero_{2i}| + k |\zero_{2i+1}| + |\one_{2i+1}|$.
		We have $|\zero_{2i}| / |\zero_{2i+1}| = 1/s_{2i+1}$ and $ |\zero_{2i+1}| = |\one_{2i+1}|$.
		We treat the underlined part as if it contained only zeros; this gives the estimate
		$$
			m(w) \ge \frac{k m(\zero_{2i+1}) + m(\one_{2i+1})}{k + 1 + 1/s_{2i+1}}
			\to
			\frac{k}{k+1} a + \frac{1}{k+1} b > a,
		$$
		as $i \to \infty$. Thus $f(\beta) \not \in B_{\chi_C}(a)$.
	\item[$\beta \in \mathcal{C}_3$:] Any prefix of $f(\beta)$ can be written as
		\begin{equation} \label{eqn:secondcase}
			w =  \underline{\zero_1^{j_0} \zero_3^{j_1} \dots} \zero_{2i-1}^{j_{i-1}} \zero_{2i+1}^{k} u,
		\end{equation}
		where $k \ge 0$ and $u$ is a prefix of $\zero_{2i+1}$.
		Recall that
		$\zero_{2i+1} = \zero_{2i}^{s_{2i}}$ and $ \zero_{2i} = \one_{2i-1} \zero_{2i-1}^{s_{2i - 1} - 1}$.
		so we further have either
		$u = \zero_{2i}^p u'$, where $u'$ is a prefix of $\one_{2i-1}$,
		or $u = \zero_{2i}^p \one_{2i-1} \zero_{2i-1}^q u'$, where $u'$ is a prefix of $\zero_{2i-1}$.
		In both cases, we have $|u'| \le l_{2i-1}$, and consequently
		\begin{equation*}
			c(u') \le 2 l_{2i-1} + p c(\zero_{2i}) + q c(\zero_{2i-1}).
		\end{equation*}
		As in the previous case, the underlined part of (\ref{eqn:secondcase}) has length $\le l_{2i-2}$.
		Thus we obtain
		\begin{multline*}
			m(w) \le
			\frac{  l_{2i-2} + j_{i-1} c(\zero_{2i-1}) + k c(\zero_{2i+1}) + 2 l_{2i-1} + p c(\zero_{2i}) + q c(\zero_{2i-1})}
			{|w|} = \\
			\frac{l_{2i-2} + 2l_{2i-1}}{|w|} +
			\frac{(j_{i-1} + q) l_{2i-1}}{|w|} m(\zero_{2i - 1}) +
			\frac{p l_{2i}}{|w|} m(\zero_{2i}) +
			\frac{k l_{2i+1}}{|w|} m(\zero_{2i + 1}).
		\end{multline*}
		Finally, since $|w| \ge j_{i-1} l_{2i-1}$ and $m(\zero_n) \le a$ for all $n$, we get
		\begin{equation*}
			m(w) \le \frac{3}{j_{i-1}} + a.
		\end{equation*}
		Our assumption $\beta \in \mathcal{C}_3$ implies that $j_i \to \infty$, and hence
		$\limsup m(w) \le a$ as the length of $w$ goes to infinity.
		Fact \ref{fact:twoergmeasures} implies that $\liminf m(w) \ge a$.
		Since $w$ is an arbitrary prefix of $f(\beta)$, we get that $f(\beta) \in B_{\chi_C}(a) = G_{\mu_0}$.
\end{itemize}
Thus $f$ is a reduction of $\mathcal{C}_3$ to $G_{\mu_0}$, which proves that the latter set is $\bP_3^0$-complete. We have completed the proof of Theorem \ref{thm:minimal_pi03}. \qed

We claim that the minimal system $X$ constructed in this section does not have SAPS.
Consider any word $w \in L(X)$ of length $l_n$.
It follows from the construction that $w$ is a subword of one of the words
$\zero_n \zero_n$,
$\zero_n \one_n$,
$\one_n \zero_n$,
$\one_n \one_n$.
Consequently, we obtain the bound $|L_{l_n}(X)| \le 4l_n + 4$. Thus
$$
	\liminf_{n \to \infty} \frac{|L_n(X)|}{n} \le 4.
$$
By a result of Cyr and Kra \cite[Theorem 1.1]{CyrKra}, this implies that there are at most 4 distinct measures in $\MT(X)$
whose set of generic points is nonempty.
Thus, for a suitable choice of $t \in [0, 1]$, the measure $\nu = t\mu_1 + (1-t) \mu_0$
will have $G_\nu = \emptyset$.
On the other hand, if $X$ had SAPS, by Corollary \ref{cor:dense-gv} the set $G_\nu$ would be dense, in particular nonempty.


\nocite{WeissBook, KLO2}

\bibliographystyle{amsplain}


\begin{thebibliography}{10}

\bibitem{AireyJacksonManceComplexityNormalPreserves}
D.~Airey, S.~Jackson, and B.~Mance, \emph{Some complexity results in the theory
  of normal numbers}, Canad. J. Math. \textbf{74} (2022), no.~1, 170--198.

\bibitem{AJKM}
Dylan Airey, Steve Jackson, Dominik Kwietniak, and Bill Mance, \emph{Borel
  complexity of sets of normal numbers via generic points in subshifts with
  specification}, Trans. Amer. Math. Soc. \textbf{373} (2020), no.~7,
  4561--4584. \MR{4127855}

\bibitem{BSS}
L.~Barreira, B.~Saussol, and J.~Schmeling, \emph{Distribution of frequencies of
  digits via multifractal analysis}, J. Number Theory \textbf{97} (2002),
  no.~2, 410--438. \MR{1942968}

\bibitem{BecherHeiberSlamanAbsNormal}
V.~Becher, P.~A. Heiber, and T.~A. Slaman, \emph{Normal numbers and the {B}orel
  hierarchy}, Fund. Math. \textbf{226} (2014), no.~1, 63--78.

\bibitem{BecherSlamanNormal}
V.~Becher and T.~A. Slaman, \emph{On the normality of numbers to different
  bases}, J. Lond. Math. Soc. (2) \textbf{90} (2014), no.~2, 472--494.

\bibitem{BerosDifferenceSet}
K.~A. Beros, \emph{Normal numbers and completeness results for difference
  sets}, J. Symb. Log. \textbf{82} (2017), no.~1, 247--257.

\bibitem{Bogachev}
V.~I. Bogachev, \emph{Measure theory. {V}ol. {I}, {II}}, Springer-Verlag,
  Berlin, 2007. \MR{2267655}

\bibitem{BowenAxiomA}
R.~Bowen, \emph{Periodic points and measures for axiom {A} diffeomorphisms},
  Trans. Amer. Math. Soc. \textbf{154} (1971), 377--397.

\bibitem{CyrKra}
Van Cyr and Bryna Kra, \emph{Counting generic measures for a subshift of linear
  growth}, Journal of the European Mathematical Society (2018).

\bibitem{DGS}
M.~Denker, C.~Grillenberger, and K.~Sigmund, \emph{Ergodic theory on compact
  spaces}, Springer, 1976.

\bibitem{Downarowicz}
Tomasz Downarowicz, \emph{Survey of odometers and {T}oeplitz flows}, Algebraic
  and topological dynamics, Contemp. Math., vol. 385, Amer. Math. Soc.,
  Providence, RI, 2005, pp.~7--37. \MR{2180227}

\bibitem{FFW}
Ai-Hua Fan, De-Jun Feng, and Jun Wu, \emph{Recurrence, dimension and entropy},
  J. London Math. Soc. (2) \textbf{64} (2001), no.~1, 229--244. \MR{1840781}

\bibitem{Garling}
D.~J.~H. Garling, \emph{Analysis on {P}olish spaces and an introduction to
  optimal transportation}, London Mathematical Society Student Texts, vol.~89,
  Cambridge University Press, Cambridge, 2018. \MR{3752187}

\bibitem{GM}
Sophie Grivaux and \'{E}tienne Matheron, \emph{Invariant measures for
  frequently hypercyclic operators}, Adv. Math. \textbf{265} (2014), 371--427.
  \MR{3255465}

\bibitem{ITV}
Godofredo Iommi, Mike Todd, and Anibal Velozo, \emph{Escape of entropy for
  countable {M}arkov shifts}, Adv. Math. \textbf{405} (2022), Paper No. 108507.
  \MR{4438058}

\bibitem{IV}
Godofredo Iommi and Anibal Velozo, \emph{The space of invariant measures for
  countable {M}arkov shifts}, J. Anal. Math. \textbf{143} (2021), no.~2,
  461--501. \MR{4299167}

\bibitem{JMR}
Stephen Jackson, Bill Mance, and Samuel Roth, \emph{A non-{B}orel special
  alpha-limit set in the square}, Ergodic Theory Dynam. Systems \textbf{42}
  (2022), no.~8, 2550--2560. \MR{4448397}

\bibitem{Kechris}
A.~Kechris, \emph{Classical descriptive set theory}, Graduate Texts in
  Mathematics, vol. 156, Springer-Verlag, New York, 1995.

\bibitem{KiLinton}
H.~Ki and T.~Linton, \emph{Normal numbers and subsets of {N} with given
  densities}, Fund. Math. \textbf{144} (1994), no.~2, 163--179.

\bibitem{KuN}
L.~Kuipers and H.~Niederreiter, \emph{Uniform distribution of sequences},
  Dover, Mineola, NY, 2006, MR0419394.

\bibitem{KLO}
Dominik Kwietniak, Martha {\L}{\c{a}}cka, and Piotr Oprocha, \emph{A panorama
  of specification-like properties and their consequences}, Dynamics and
  numbers, Contemp. Math., vol. 669, Amer. Math. Soc., Providence, RI, 2016,
  pp.~155--186. \MR{3546668}

\bibitem{KLO2}
Dominik Kwietniak, Martha \L\c{a}cka, and Piotr Oprocha, \emph{Generic points
  for dynamical systems with average shadowing}, Monatsh. Math. \textbf{183}
  (2017), no.~4, 625--648. \MR{3669782}

\bibitem{KU}
Dominik Kwietniak and Martha Ubik, \emph{Topological entropy of compact
  subsystems of transitive real line maps}, Dyn. Syst. \textbf{28} (2013),
  no.~1, 62--75. \MR{3040767}

\bibitem{Lind}
D.~A. Lind, \emph{Dynamical properties of quasihyperbolic toral automorphisms},
  Ergodic Theory Dynam. Systems \textbf{2} (1982), no.~1, 49--68. \MR{684244}

\bibitem{Marcus}
Brian Marcus, \emph{A note on periodic points for ergodic toral automorphisms},
  Monatsh. Math. \textbf{89} (1980), no.~2, 121--129. \MR{572888}

\bibitem{Olsen}
L.~Olsen, \emph{Applications of multifractal divergence points to sets of
  numbers defined by their {$N$}-adic expansion}, Math. Proc. Cambridge Philos.
  Soc. \textbf{136} (2004), no.~1, 139--165. \MR{2034019}

\bibitem{Oxtoby}
John~C. Oxtoby, \emph{Ergodic sets}, Bull. Amer. Math. Soc. \textbf{58} (1952),
  116--136. \MR{47262}

\bibitem{PS}
C.~E. Pfister and W.~G. Sullivan, \emph{Large deviations estimates for
  dynamical systems without the specification property. {A}pplications to the
  {$\beta$}-shifts}, Nonlinearity \textbf{18} (2005), no.~1, 237--261.
  \MR{2109476}

\bibitem{Sarig}
Omri~M. Sarig, \emph{Phase transitions for countable {M}arkov shifts}, Comm.
  Math. Phys. \textbf{217} (2001), no.~3, 555--577. \MR{1822107}

\bibitem{WeissBook}
Benjamin Weiss, \emph{Single orbit dynamics}, CBMS Regional Conference Series
  in Mathematics, vol.~95, American Mathematical Society, Providence, RI, 2000.
  \MR{1727510}

\end{thebibliography}
\providecommand{\bysame}{\leavevmode\hbox to3em{\hrulefill}\thinspace}
\providecommand{\MR}{\relax\ifhmode\unskip\space\fi MR }
\providecommand{\MRhref}[2]{%
  \href{http://www.ams.org/mathscinet-getitem?mr=#1}{#2}
}
\providecommand{\href}[2]{#2}

\end{document}